\documentclass[11pt]{amsart}                                        

\usepackage{a4wide}

\usepackage{latexsym}
\usepackage{amscd}                           
\usepackage{graphics, graphicx, psfrag}
\usepackage{amsmath}            
\usepackage{amssymb}
\usepackage{mathrsfs}
\usepackage{bbm}
\usepackage{amsthm}
\usepackage{enumerate}
\usepackage{color}
\usepackage{ulem}
\input xy
\xyoption{all}

\def\1{\:\!}
\def\2{\;\!}
\def\-2{\:\!\!}
\def\s{\smallskip}
\def\m{\medskip}
\def\b{\bigskip}
\def\ni{\noindent}
\def\gve{\varepsilon}

\newcommand{\N}{\mathbb{N}}                     % the natural numbers
\newcommand{\Z}{\mathbb{Z}}                     % the integer numbers
\newcommand{\R}{\mathbb{R}}                     % the real line
\newcommand{\C}{\mathbb{C}}                     % the complex plane
\newcommand{\T}{\mathbb{T}}                     % the torus
\newcommand{\D}{\mathbb{D}}                     % the unit disk 
\newcommand{\A}{\mathbb{A}}                     % \A_H = action functional

\newcommand{\cd}{\mathscr{D}}

\newcommand{\ch}{\mathscr{H}}
\newcommand{\cj}{\mathscr{J}}
\newcommand{\ck}{\mathscr{K}}
\newcommand{\cL}{\mathscr{L}}

\newcommand{\cp}{\mathscr{P}}

\newcommand{\cu}{\mathscr{U}}

\newcommand{\set}[2]{\left\{{#1}\mid{#2}\right\}}       % the set
\newcommand{\supp}{\mathrm{supp\,}}             % support
\newcommand{\spec}{\mathrm{spec}\2}               % spectrum
\newcommand{\loc}{\operatorname{loc}}         
\newcommand{\ev}{\operatorname{ev}}         
\newcommand{\ess}{\operatorname{ess}}  
  
\newcommand{\Ham}{\operatorname{Ham}}  
\newcommand{\Int}{\operatorname{Int}}  
  
\def\PSS{ {\scriptscriptstyle\operatorname{PSS}}}

\newcommand{\HF}{\operatorname{HF}}

\newcommand{\id}{\operatorname{id}}    

\newcommand{\proofend}{\hspace*{\fill} $\Box$\\}

\newtheorem{thm}{Theorem}[section]      % numbered along with sections
\newtheorem{cor}[thm]{Corollary}        % numbered along with Theorem
\newtheorem{lem}[thm]{Lemma}            % numbered along with Theorem
\newtheorem{prop}[thm]{Proposition}     % numbered along with Theorem
\newtheorem{defn}[thm]{Definition}      % numbered along with Theorem
\newtheorem{rem}[thm]{Remark}           % numbered along with Theorem  
\newtheorem{ex}[thm]{Example}           % numbered along with Theorem
\newtheorem{open}[thm]{Open Problem}    % numbered along with Theorem

\def\red#1{{\textcolor{red}{#1}}}

\begin{document}

\title[A simple construction of an action selector]{A simple construction of an action selector on aspherical symplectic manifolds}

\author{Alberto Abbondandolo}
\thanks{The research of A.\ Abbondandolo is supported by the SFB/TRR 191 
   ‘Symplectic Structures in Geometry, Algebra and Dynamics’, funded by the Deutsche Forschungsgemeinschaft.}
\address{Alberto Abbondandolo, Fakult\"at f\"ur Mathematik, Ruhr-Universit\"at Bochum}
\email{alberto.abbondandolo@rub.de}
\author{Carsten Haug}  
\address{Carsten Haug,
Institut de Math\'ematiques,
Universit\'e de Neuch\^atel}
\email{carsten.haug@unine.ch}
\author{Felix Schlenk}  
\thanks{The research of F.\ Schlenk is supported by the SNF grant 200021-181980/1}
\address{Felix Schlenk,
Institut de Math\'ematiques,
Universit\'e de Neuch\^atel}
\email{schlenk@unine.ch}



\maketitle

%\tableofcontents

\begin{abstract}
We construct an action selector on aspherical symplectic manifolds that are closed or convex.
Such selectors have been constructed by Matthias Schwarz 
using Floer homology. The construction we present here is simpler and uses only Gromov compactness.
\end{abstract}

\section{Introduction}

Hamiltonian systems on symplectic manifolds tend to have many periodic
orbits. 
The ``actions'' of these orbits form an invariant for the Hamiltonian system. 
The set of actions can be very large, however. 
To get useful invariants, one selects for each Hamiltonian function 
just one action value by some minimax procedure: 
A so-called action selector associates to every time-periodic
Hamiltonian function on a symplectic manifold
the action of a periodic orbit of its flow in a continuous way. 
For this one needs compactness assumptions on either the symplectic manifold or the support 
of the Hamiltonian vector field.
The mere existence of an action selector has many applications
to Hamiltonian dynamics and symplectic topology:
It readily yields a symplectic capacity and thus implies Gromov's non-squeezing theorem, 
implies the almost existence of closed characteristics on displaceable hypersurfaces
and in particular the Weinstein conjecture for displaceable energy surfaces
of contact type, 
often proves the non-degeneracy of Hofer's metric and its unboundedness, 
etc., 
see for instance~\cite{FrGiSch05, FrSch07, hz94, Ost03, Sch00, Vi92}
and Section~\ref{s:3app} below.

Action selectors were first constructed for the standard symplectic vector
space~$(\R^{2n},\omega_0)$ by Viterbo~\cite{Vi92}, and by Hofer--Zehnder~\cite{hz94}
who built on earlier work by Ekeland--Hofer~\cite{EkHo89}. 
For more general symplectic manifolds~$(M,\omega)$, action selectors were obtained, 
up until now, only by means of Floer homology:
For symplectically aspherical symplectic manifolds
(namely those for which $[\omega] |_{\pi_2(M)}=0$),
Schwarz~\cite{Sch00} constructed the so-called PSS selector when $M$ is closed, 
and his construction was adapted to convex symplectic manifolds in~\cite{FrSch07}.
We refer to Appendix~A of~\cite{FrGiSch05} for a short description of these selectors.
For some further classes of symplectic manifolds and Hamiltonian functions, 
the PSS selector was constructed in~\cite{La13, Oh05, Ush08}.

In this paper we give a more elementary construction of an action selector for closed or 
convex symplectically aspherical manifolds.
Our construction uses only results from Section~6.4 of the text book~\cite{hz94}
by Hofer and Zehnder, that rely on Gromov compactness and rudimentary Fredholm theory,
but on none of the more advanced tools in the construction of Floer homology
(such as exponential decay, the spectral flow, unique continuation, gluing, or transversality).
In this way, the three basic properties of an action selector 
(spectrality, continuity and local non-triviality) are readily established by rather straightforward proofs, since the only tool at our hands is the compactness property of certain spaces of holomorphic cylinders.

After recollecting known results in Section~\ref{s:not},
we give the construction of our action selector for closed symplectically aspherical manifolds
in Section~\ref{s:minimal}.
In Section~\ref{s:convex} we adapt this construction to convex symplectically aspherical manifolds.
Examples are cotangent bundles and their fiberwise starshaped subdomains, on which most of classical mechanics
takes place. 
In Section~\ref{s:axiom} we show that the three basic properties of the action selector imply many further properties,
and in Section~\ref{s:3app} we illustrate by three examples how any action selector yields simple proofs of results in 
symplectic geometry and Hamiltonian dynamics. 
In Section~\ref{s:further} we sketch some variations of our construction and address open problems.

\m \ni
{\bf Idea of the construction.}
In the rest of this introduction we outline the construction of our action selector
on a closed symplectically aspherical manifold~$(M,\omega)$.
Denote by $\T = \R / \Z$ the circle of length~$1$. 
Recall that the Hamiltonian action functional on the space of contractible loops 
$C^{\infty}_{\mathrm{contr}} (\T,M)$ associated to a Hamiltonian function $H \in C^\infty(\T \times M, \R) =: \ch (M)$ 
is given by 
\[
\A_H (x) \,:=\, \int_{\D} \bar{x}^* (\omega) + \int_{\T} H(t,x(t))\, dt,
\]             
where $\bar{x}\in C^{\infty}(\D,M)$ is such that $\bar{x}|_{\partial \D} = x$.
The critical points of~$\A_H$ are the contractible $1$-periodic 
solutions of the Hamiltonian equation 
\[
\dot x (t) \,=\, X_H (t,x(t)),
\]
where the vector field~$X_H$ is defined by $\omega(X_H,\cdot) = dH$,  
and the set of critical values of~$\A_H$ is called the
action spectrum of~$H$ and denoted by~$\spec (H)$. An action selector should select an element of $\spec (H)$ in a monotone and continuous way, with respect to the usual order relation and to some reasonable topology on the space of Hamiltonians.

A first idea  for defining an action selector is
to boldly take the smallest action value of a 1-periodic orbit, 
\[
\sigma (H) \,:=\, \min \spec (H) .
\]
Since $\spec (H)$ is a compact subset of~$\R$, this definition makes sense,
and yields an invariant with the spectral property. 
However, this invariant is not very useful, since it fails to be continuous
and monotone, two crucial properties for applications. 
To see why, consider radial functions 
\[
H_f(z) \,:=\, f (\pi |z|^2) \quad  \mbox{ on }\, \R^{2n},
\]
where $f \colon [0,+\infty) \to \R$ is a smooth function with compact support. 
For an arbitrary symplectic manifold, such functions can be constructed in a Darboux chart and then be extended 
by zero to the whole manifold.
The critical points of~$\A_H$ are the origin and the (Hopf-)circles on those 
spheres that have radius~$r$ with $s = \pi r^2$ and $f'(s) \in \Z$;
at such a critical point~$x$ the value of the action is
\begin{equation} \label{e:Af}
\A_{H_f}(x) \,=\, f(s) - s \2 f'(s) ,
\end{equation}
see the left drawing in Figure~\ref{fig.radial}.
Now take the profile functions $f, f_+, f_-$ as in the right drawing:
$f' \in [0,1]$ and $f'(s) = 1$ for a unique~$s$,
while $f_-, f_+$ are $C^\infty$-close to~$f$ 
and satisfy $f_- \leq f \leq f_+$ and $f'_-, f'_+ \in [0,1)$.
Then the formula~\eqref{e:Af} shows that 
$\sigma (H_f)$ is much smaller than $\sigma (H_{f_-}) \approx \sigma (H_{f_+})$,
whence $\sigma$ is neither continuous nor monotone.
Or take $g$ with $|g|$ very small and very steep. 
Then $\sigma (H_g)$ is much smaller than $\sigma (H_f)$, whence monotonicity fails drastically.

\begin{figure}[h]   
 \begin{center}
  \psfrag{s}{$s$}
  \psfrag{sr}{$s=\pi r^2$}  \psfrag{fs}{$f(s)$}  \psfrag{fs-}{$f(s)-sf'(s)$}  \psfrag{f}{$f$}  \psfrag{f+}{$f_+$}
  \psfrag{f-}{$f_-$}  \psfrag{g}{$g$}  \psfrag{sf}{$\sigma (H_f)$}  \psfrag{sf-}{$\sigma (H_{f_-})$}  \psfrag{sf+}{$\sigma (H_{f_+})$}
  \psfrag{sg}{$\sigma (H_g)$}
  \leavevmode\includegraphics{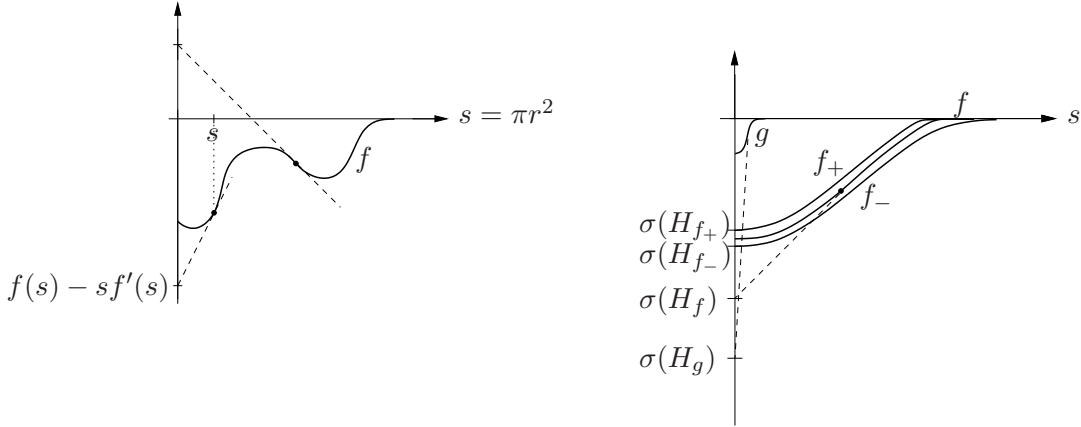}
 \end{center}
 \caption{Radial functions and their minimal spectral values}   \label{fig.radial}
\end{figure}

\m
The above discussion shows that the continuous, or monotone, selection of an action 
from $\spec (H)$ must be done by some kind of minimax procedure involving more information on the action functional than the mere knowledge of its critical values. 
This was done for the Hofer--Zehnder selector by minimax over a uniform minimax family, 
and for the Viterbo selector and the PSS selector by a homological minimax.
Our minimax will be over certain spaces of perturbed holomorphic cylinders.

\begin{figure}[h]   
 \begin{center}
 \psfrag{x}{$x$}  \psfrag{y}{$y$}  \psfrag{q}{$q_h$}  \psfrag{0}{$(0,0)$}  \psfrag{p1}{$p_1$}  \psfrag{p_2}{$p_2$}
  \leavevmode\includegraphics{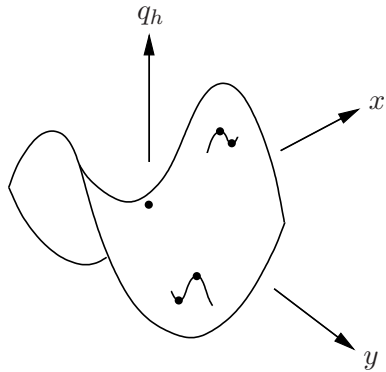}
 \end{center}
 \caption{A perturbed quadratic form $q_h$}  
 \label{fig.saddle}
\end{figure}

To introduce our construction, we first look at a toy model: 
Consider the quadratic form $q(x,y) = x^2-y^2$ on~$\R^2$
and its perturbations
\[
q_h \,=\, q + h
\]
where $h$ is a compactly supported function on $\R^2$. 
Here, the indefinite quadratic form~$q$ models the symplectic action and 
the compactly supported function~$h$ models the Hamiltonian term in~$\A_H$,
cf.\ \cite[\S 3.3]{hz94}.  
If $h=0$, the only critical point of~$q_h$ is the origin, with critical value~$0$.
If $h$ consists, for instance, of two little positive bumps, 
one centered at $(1,0)$ and one at~$(0,1)$,
then the graph of~$q_h$ looks as in Figure~\ref{fig.saddle}.
A continuous selection of critical values $h \mapsto \sigma(h)$ should, in our example, choose again~$0$,
by somehow discarding the four new critical values.

In this finite dimensional example, one could define an action selector by the minimax formula
\[
\sigma(h) \,=\, \inf \max_Y q_h,
\]
where the infimum is over the space of all images~$Y$ of continuous maps $\R \rightarrow \R^2$ 
that are compactly supported perturbations of the embedding $y \mapsto (0,y)$. 
Monotonicity in~$h$ is clear from the definition, and spectrality can be proved by standard deformation arguments using the negative gradient flow of~$q_h$. The definition of the Hofer--Zehnder action selector 
(see \cite[Section 5.3]{hz94}) is based on a similar idea and uses the fact that the Hamiltonian action functional for loops in~$\R^{2n}$ has a nice negative gradient flow.

Alternatively, one can fix a very large number~$c$ such that the sublevel $\{q_h<-c\}$ coincides with the sublevel $\{q<-c\}$ and define
the same critical value $\sigma (h)$ as
\[
\inf \left\{ a \in \R \mid \mbox{the image of } 
   i^a_* \colon H_1 ( \{q_h<a\},\{q<-c\} ) \rightarrow H_1 ( \R^2,\{q<-c\} ) 
   \mbox{ is non-zero} \right\},
\]
where the map $i^a$ is the inclusion
\[
i^a \colon \left( \{q_h<a\},\{q<-c\} \right) \hookrightarrow \left( \R^2,\{q<-c\} \right)
\]
and we are using the fact that
\[
H_1 ( \R^2,\{q<-c\} ) \cong \Z.
\]
Viterbo's definition of an action selector for compactly supported Hamiltonians on~$\R^{2n}$ uses a similar construction, which is applied to suitable generating functions, see~\cite{Vi92}. The Floer homological translation of this second definition is, in turn, at the basis of Schwarz's construction of an action selector for symplectically aspherical manifolds, see~\cite{Sch00}, 
and of all its subsequent generalizations.

Here, we would like to define an action selector $\sigma(h)$ using only spaces of bounded negative gradient flow lines: 
In the case of the Hamiltonian action functional~$\A_H$, these will correspond to finite energy solutions of the Floer equation, which have good compactness properties. A first observation is that the knowledge of the space of all bounded negative gradient flow lines of~$q_h$ is not enough for defining an action selector.
Indeed, it is easy to perturb $q$ on a small disc disjoint from the origin in such a way that the negative 
gradient flow lines of $q_h$ look like in Figure~\ref{fig.nose}:
A new degenerate critical point~$z$ is created, and the constant orbits at $(0,0)$ and at~$z$
are the only bounded negative gradient flow lines. 
But since $q_h(z)$ could be either positive or negative, the set $\{(0,0),z\}$ contains too little information for us 
to conclude that the value of the action selector should be $q_h(0,0)=0$.

\begin{figure}[h]   
 \begin{center}
 \psfrag{x}{$x$}  \psfrag{y}{$y$}  \psfrag{z}{$z$}  \psfrag{0}{$(0,0)$}
  \leavevmode\includegraphics{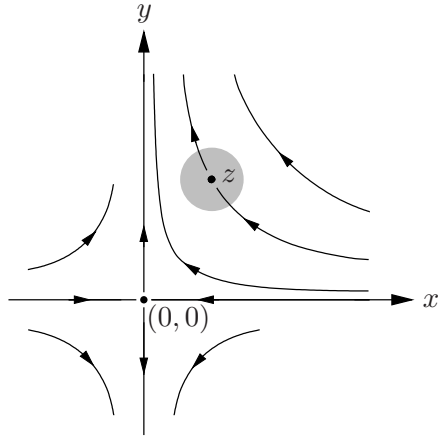}
 \end{center}
 \caption{The only bounded gradient flow lines are the constant orbits at $(0,0)$ and~$z$. }  \label{fig.nose}
\end{figure}

If, however, we are allowed to deform the function $q_h$, we can use bounded gradient flow lines to define 
an action selector that identifies the 
lowest critical value that ``cannot be shaken off''. More precisely, take a family~$\{h^s\}_{s\in \R}$ of compactly supported functions
such that $h^s = h$ for $s$ small and $h^s =0$ for $s$ large,
and look at the space~$\cu (h^s)$ of bounded solutions of the non-autonomous gradient equation
\[
\dot u(s) \,=\, -\nabla q_{h^s}(u(s)) , \qquad s \in \R .
\]
The boundedness of $u$ is equivalent to bounded energy 
\[
E(u) \,:=\, \int_\R |\nabla q_{h^s}(u(s))|^2 \,ds \,=\, 
\lim_{s \to -\infty} q_{h^s} (u(s)) - \lim_{s \to +\infty} q_{h^s} (u(s)) + \int_{\R} \frac{\partial h^s}{\partial s}(u(s))\, ds \,<\, \infty ,
\] 
or, since $h^s =h$ in the first limit and $h^s=0$ in the second limit, 
to the fact that $u(s)$ is asymptotic for $s\rightarrow -\infty$ to the following critical level of $q_h$
\[
q_h^- (u) \,:=\, \lim_{s \to -\infty} q_h (u(s))
\]
and for all $s$ large lies on the $x$-axis and converges for $s\rightarrow +\infty$ to the origin 
(the only critical point of $q$).
The number
\[
\min_{u \in \cu (h^s)} q_h^-(u) 
\]
is the lowest critical value of~$q_h$
from which a bounded $h^s$-negative gradient flow line starts.

In our example from Figure~\ref{fig.saddle}, 
if we take $h^s = \beta (s) \,h$ with a cut-off function~$\beta$,
then $\cu (h^s)$ contains no flow line~$u$ emanating from the two low critical points $p_1$ or~$p_2$ near~$(0,1)$.
On the other hand, it is easy to construct a family~$h^s$ that has 
a negative gradient line~$u(s)$ 
that converges to $p_1$ for $s\rightarrow -\infty$ and to the origin for $s\rightarrow +\infty$.
To be sure that we discard all inessential critical values, 
we therefore set
\[
\sigma (h) \,:=\, \sup_{h^s} \min_{u \in \cu (h^s)} q_h^-(u).
\]
In the example, it is quite clear that for every deformation $h^s$ there exists a flow line in~$\cu (h^s)$
emanating from the critical point~$(0,0)$, that is, 
$\sigma (h) =0$ as it should be.
In general, it is not hard to see that $\sigma (h)$ is a critical value of~$q_h$
that depends continuously and in a monotone way on~$h$.

The number $\sigma (h)$ is the lowest critical value~$c$ of~$q_h$
such that for every deformation~$h^s$ of~$h$ there exists a bounded flow line $u \in \cu (h^s)$
starting at a critical level not exceeding~$c$. Equivalently, $\sigma(h)$ is the highest critical value~$c$ of~$q_h$ 
such that for every critical level $c'<c$
there exists a deformation~$h^s$ of~$h$ such that all flow lines of $q_{h^s}$ starting at
level~$c'$
are unbounded.
That is: the whole critical set strictly below~$c$ can be shaken off.

\m
Imitating the above construction, 
and inspired by the proof of the degenerate Arnol'd conjecture in~\cite[\S 6.4]{hz94}, 
we can define an action selector for 1-periodic Hamiltonians on a closed symplectically aspherical manifold~$(M,\omega)$ in the following way. Given $H\in C^{\infty}(\T\times M)$ we consider $s$-dependent Hamiltonians $K$ in $C^{\infty}(\R \times \T \times M )$  such that $K(s,\cdot,\cdot)=H$ for $s$~small and $K(s,\cdot,\cdot)=0$ for $s$~large.  Following Floer's interpretation of the $L^2$-gradient flow of the action functional, 
we consider the space~$\cu(K)$ of solutions $u \in C^{\infty}(\R \times \T,M)$
of Floer's equation
\begin{equation}
\label{floers}
\partial_s u + J(u) \bigl( \partial_t u - X_{K} (s,t,u) \bigr) = 0
\end{equation}
that have finite energy 
\[
E(u) = \int_{\R \times \T} |\partial_su|_{J}^2 < \infty.
\]
Here, $J$ is a fixed $\omega$-compatible almost complex structure on~$TM$ and $|\cdot|_J$ is the induced Riemannian norm.
The space $\cu (K)$ is $C^\infty_{\loc}$-compact by Gromov's compactness theorem.
Now define the function
\[
a^-_H \colon \cu(K) \to \R, \qquad a^-_H (u) := \lim_{s \to -\infty} \A_H(u(s))
\]
and finally define the action selector of~$H$ by
\[
A_{J}(H) \,:=\, \sup_{K} \min_{u \in \cu(K)} a_H^-(u),
\]
where the supremum is taken over all deformations $K$ of~$H$ as above.
The number $A_{J} (H)$ is the smallest essential action of~$H$ in the following sense:
It is the lowest critical value~$c$ of~$\A_H$ 
(that is, the lowest action of a contractible 1-periodic orbit of~$H$)
such that for every deformation~$K$ of~$H$ there exists a finite energy solution of Floer's equation
for $K$ and~$J$ that starts at a critical level~$\leq c$.
For another characterization of $A_{J}(H)$, see Section~\ref{ss:equivalent}.

In our finite dimensional model, 
we could have allowed for a larger class of deformations of the gradient flow of~$q_h$,
by looking at families~$h^s$ that for $s$ large do not depend on $s$ but are not necessarily zero,
and by taking the gradient with respect to any family~$g_s$ of Riemannian metrics
that depend on~$s$ on a compact interval.
In the symplectic setting, the role of Riemannian metrics is played by
$\omega$-compatible almost complex structures. 
We may thus modify the above definition
by looking at functions~$K$ with $K(s,\cdot,\cdot)=H$ for $s$ small and $K(s,\cdot,\cdot)$ independent of $s$ for $s$ large,
and at families~$J^s$ of $\omega$-compatible almost complex structures
that depend on~$s$ on a compact interval.
In Sections~\ref{s:not}--\ref{s:convex}, we shall construct
an action selector~$A(H)$ by using these larger families of deformations.
This has the advantage that $A(H)$ is manifestly independent of the choice of~$J$.
It will be clear from the analysis of~$A(H)$ that $A_{J}(H)$ is also an action selector, 
cf.\ Section~\ref{s:smaller}.

There are also action selectors relative to closed Lagrangian submanifolds, 
that have many applications in the study of these important submanifolds. 
Such selectors were first constructed by Viterbo~\cite{Vi92} and Oh~\cite{Oh97}
for Hamiltonian deformations of the zero-section of cotangent bundles, 
and then in more general settings by Leclercq~\cite{Le08} and Leclercq--Zapolsky~\cite{LeZa18}.
Except for Viterbo's generating function approach, all these constructions are based on Lagrangian Floer homology. 
Our elementary construction of an action selector can also be carried out for closed Lagrangian submanifolds~$L$ 
under the assumption that $[\omega]$ vanishes on~$\pi_2(M,L)$.
We shall focus on the absolute case, however, leaving the necessary adaptations to the interested reader.

\section{Notations, conventions and known results}
\label{s:not}

Let $(M,\omega)$ be a closed symplectic manifold such that $[\omega]|_{\pi_2(M)}=0$.
We assume throughout that $M$ is connected. 
We denote by $X_H$ the Hamiltonian vector field associated to a Hamiltonian $H \in C^{\infty}(M)$, that is
\[
\omega(X_H,\cdot) = dH.
\]
Let $\T = \R / \Z$ be the circle of length~$1$.
The Hamiltonian action functional on the space of contractible loops 
$C^{\infty}_{\mathrm{contr}} (\T,M)$ associated to a time-periodic Hamiltonian $H \in C^{\infty}(\T \times M)$ 
has the form
\[
\A_H(x) := \int_{\D} \bar{x}^* (\omega) + \int_{\T} H(t,x(t))\, dt,
\]
where $\bar{x}\in C^{\infty}(\D,M)$ is an extension of the loop $x$ to the closed disk~$\D$, 
that is $\bar{x}|_{\partial \D}=x$; here we are identifying $\partial \D$ and~$\T$ in the standard way. 
The first integral does not depend on the choice of the extension $\bar{x}$ of~$x$ because $[\omega]$ 
vanishes on~$\pi_2(M)$. The critical points of~$\A_H$ are precisely the elements of~$\cp(H)$, the set of contractible 1-periodic orbits of~$X_H$. By the Ascoli--Arzel\`a theorem, $\cp(H)$ is a compact subset 
of $C^{\infty}_{\mathrm{contr}}(\T,M)$. 

The space $C^{\infty}(\R\times \T,M)$ is endowed with the $C^{\infty}_{\loc}$-topology, 
which is metrizable and complete. 
We shall identify $C^{\infty}(\R\times \T,M)$ with $C^{\infty}(\R, C^{\infty}(\T,M))$, and we use the notation
\[
u(s) = u(s,\cdot) \in C^{\infty}(\T,M), \qquad \forall \2 s\in \R.
\]
The additive group $\R$ acts on $C^{\infty}(\R\times \T,M)$ by translations
\[
(\sigma,u) \mapsto \tau_{\sigma} u,\qquad \mbox{where } (\tau_{\sigma} u)(s):= u(\sigma+s).
\]
Let $J$ be a smooth $\omega$-compatible almost complex structure on~$M$, meaning that
$$
g_J(\xi,\eta) := \omega(J \xi,\eta), \qquad \forall \2 \xi, \eta\in T_x M, \; \forall \2 x\in M,
$$
is a Riemannian metric on $M$. The associated norm is denoted by $|\cdot|_J$.
The $L^2$-negative gradient equation for the functional $\A_H$ is the Floer equation
\begin{equation}
\label{floer}
\partial_s u + J(u) \bigl( \partial_t u - X_H(t,u) \bigr) = 0.
\end{equation}
If $u$ is a solution of~\eqref{floer}, then the function $s \mapsto \A_H(u(s,\cdot))$ 
is non-increasing and
\begin{equation} \label{e:E}
\lim_{s\to -\infty} \A_H(u(s,\cdot)) - \lim_{s\to +\infty} \A_H(u(s,\cdot)) = E(u) := \int_{\R \times \T} \bigl|\partial_s u \bigr|_J^2\, ds \,dt.
\end{equation}
The quantity $E(u)$ defined above is called energy of the cylinder $u$.
Any $x\in \mathscr{P}(H)$ defines a stationary solution $u(s,t):=x(t)$ of~\eqref{floer}, 
which has zero energy and is called a trivial cylinder.

Now let $H \in C^{\infty}(\R \times \T \times M,\R)$ be such that $\partial_s H$, the partial derivative of~$H$ with respect to the first variable, has compact support and set
\[
H^-(t,x) := H(-s,t,x) \quad \mbox{and} \quad H^+(t,x) := H(s,t,x) \quad \mbox{for $s$ large}.
\] 
Further, let $J=\{J^s\}$ be a smooth $s$-dependent family of $\omega$-compatible almost complex structures 
such that $\partial_s J$ has compact support, and set
\[
J^-(x) := J^{-s}(x) \quad \mbox{and} \quad J^+(x) := J^s(x) \quad \mbox{for $s$ large}.
\]
If $u$ solves the $s$-dependent Floer equation
\begin{equation}
\label{floer-s}
\partial_s u + J^s(u) \bigl( \partial_t u - X_H(s,t,u) \bigr) = 0,
\end{equation}
then the energy identity reads:
\begin{equation}
\label{act-en}
\A_{H(s_0,\cdot,\cdot)} (u(s_0)) - \A_{H(s_1,\cdot,\cdot)} (u(s_1)) = 
\int_{[s_0,s_1] \times \T} \bigl| \partial_s u \bigr|_{J^s}^2\, ds \, dt - \int_{[s_0,s_1] \times \T} \partial_s H(s,t,u(s,t)) \, ds \, dt,
\end{equation}
for every $s_0<s_1$.
It follows that the function $s\mapsto \A_{H(s,\cdot,\cdot)}(u(s))$ is non-increasing 
on a neighborhood of~$-\infty$ and on a neighborhood of~$+\infty$, and that
\[
\lim_{s\to -\infty} \A_{H^-}(u(s,\cdot)) - \lim_{s\to +\infty} \A_{H^+}(u(s,\cdot)) = E(u) - \int_{\R \times \T} \partial_s H(s,t,u(s,t)) \, ds \, dt 
\]
where the energy $E(u)$ is defined as in~\eqref{e:E}, but with an $s$-dependent~$J$:
\[
E(u) = E_J(u) := \int_{\R \times \T} \bigl|\partial_s u \bigr|_{J^s}^2\, ds \,dt.
\]
Set
\[
\cu(H,J) \,:=\, \set{ u \in C^{\infty}(\R \times \T,M) }{u \mbox{ is a solution of~\eqref{floer-s} with } E(u)<\infty}.
\]
We recall that a subset $\cu$ of $C^{\infty}(\R\times \T,M)$ is said to be bounded if for every 
multi-index $\alpha\in \N^2$, $|\alpha|\geq 1$, there holds
\[
\sup_{u \in  \cu} \sup_{(s,t)\in \R\times \T}  | \partial_s^{\alpha_1} \partial_t^{\alpha_2} u(s,t) |_J < \infty.
\]
Bounded subsets are relatively compact in the $C^{\infty}_{\loc}$-topology. The next result is a special instance of Gromov compactness.

\begin{prop}
\label{comp}
Let $H = \{H^s\}$ and $J=\{J^s\}$ be as above.
\begin{enumerate}[\rm (i)]
\item The set $\cu(H,J)$ is a compact subset of $C^{\infty}(\R\times \T,M)$.
\item For $u \in \cu(H,J)$ the set
\[
\alpha\mbox{-}\lim (u) \,:=\, \set{ \lim_{n \to \infty} \tau_{s_n} u }{s_n \to -\infty 
\mbox{ is such that } \tau_{s_n} u \mbox{ converges}} 
\]
is a non-empty subset of~$\cu(H^-,J^-)$ and consists of trivial cylinders of the form $v(s,t):=x(t)$, 
for some $x \in \cp(H^-)$ with action
\[
\A_{H^-}(x) = \lim_{s\rightarrow -\infty} \A_{H^-}(u(s)).
\]
\end{enumerate}
\end{prop}

\begin{proof}[Outline of the proof.]
Statement (i)  is proved in Corollary~1 and Proposition~11 in Section~6.4 of~\cite{hz94} for the case that 
$H$ and~$J$ do not depend on~$s$. That proof readily generalizes to our situation. 
Statement~(ii) can be obtained by adapting Propositions~8 and~9 in \cite[\S 6.3]{hz94} and by
using Lemma~2 in \cite[\S 6.4]{hz94}. We nevertheless sketch the main steps of the proof of both statements, 
since we wish to make clear which tools are actually used. See \cite[\S 3]{Ha18} for more details.

One starts by proving that for every $c\geq 0$ the set
\[
\cu_c(H,J) \,:=\, \{ u\in \cu(H,J) \mid E(u)\leq c\}
\]
is compact. The uniform boundedness of the first derivatives requires a bubbling-off analysis, that uses the assumption that $[\omega]$ vanishes on~$\pi_2(M)$ and the uniform bound on the energy. 
Once uniform bounds on the first derivatives have been established, the bounds on all higher derivatives follow from elliptic bootstrapping. This shows that $\cu_c(H,J)$ is bounded in $C^{\infty}(\R \times \T, M)$. 
By the lower semicontinuity of the energy, that is
\[
u_n \rightarrow u \mbox{ in } C^{\infty}(\R\times \T,M) \qquad \Longrightarrow \qquad E(u) \leq \liminf_{n\rightarrow \infty} E(u_n),
\]
the set $\cu_c(H,J)$ is also closed in $C^{\infty}(\R \times \T, M)$, and hence compact.

Statement (i) will thus follow from the fact that $\cu(H,J)=\cu_c(H,J)$ when $c$ is large enough. 
In order to prove the latter fact, we need to address statement~(ii). Let $u\in \cu(H,J)$. 
That the set $\alpha$-$\lim(u)$ is not empty follows from the fact that the set 
\[
\{\tau_s u \mid s\in \R\}
\]
is relatively compact in $C^{\infty}(\R \times \T, M)$, by the same argument sketched above. 
Now assume that $v = \lim_{n \to \infty} \tau_{s_n}u$ with $s_n \to -\infty$. Since $v_n := \tau_{s_n} u$ solves 
the equation
\[
\partial_s v_n + (\tau_{s_n}J)(v_n) \bigl( \partial_t v_n - X_{\tau_{s_n} H}(s,t,v_n) \bigr) = 0,
\]
and since $\tau_{s_n} H$ converges to $H^-$ and $\tau_{s_n}J$ converges to~$J^-$, the limit~$v$ 
is a solution of the $s$-independent Floer equation defined by~$H^-$ and~$J^-$. 
Moreover, since
\[
\int_{[-T,T] \times \T} |\partial_sv|^2_{J^-} \,ds\,dt \,=\, 
\lim_{n \to \infty} \int_{[-T,T] \times \T} |\partial_s v_n|^2_{\tau_{s_n} J} \,ds\,dt \,\leq\,
\liminf_{n \to \infty} E_{\tau_{s_n} J} (v_n)
\]
for every $T>0$
and since
$E_{\tau_{s_n}J}(v_n) = E_J(u)$ for all~$n$, we have
\[
E_{J^-}(v) \,\leq\, \liminf_{n \to \infty} E_{\tau_{s_n}J}(v_n) \,=\, E_J(u) .
\]
Hence $v \in \cu(H^-,J^-)$, and it remains to show that $v$ is a trivial cylinder for~$H^-$. 
Consider the function
\[
a_{H^-} \colon C^{\infty}(\R\times \T,M) \to \R, \qquad a_{H^-}(w) = \A_{H^-}(w(0)) .
\]
Since $H(s,\cdot,\cdot)=H^-$ for $s\leq -S$, where $S$ is a sufficiently large number, 
the function 
\[
s\mapsto a_{H^-}(\tau_s u) = \A_{H^-}(u(s))
\]
is non-increasing on the interval $(-\infty,-S]$. Since $u$ has finite energy, this function is also bounded, 
and hence converges to some real number~$a$ for $s\rightarrow -\infty$. From the continuity of~$a_{H^-}$ 
we deduce that $a_{H^-}(v)=a$ for all $v \in \mbox{$\alpha$-$\lim(u)$}$. The latter set is clearly invariant 
under the action of~$\tau_s$, so we have that
\[
a_{H^-}(\tau_s v) = a_{H^-}(v) = a = \lim_{s\rightarrow -\infty} \A_{H^-}(u(s))
\]
for all $s\in \R$. The energy identity for $v$ then forces $v$ to be a trivial cylinder $v(s,t)=x(t)$ 
of action $\A_{H^-}(x)=a$. This concludes the proof of~(ii).

By the energy identity \eqref{act-en}, each $u\in \cu(H,J)$ has then the uniform energy bound
\begin{equation}
\label{unifen}
E(u) \leq \max_{x\in \mathscr{P}(H^-)} \A_{H^-}(x) -  
          \min_{x\in \mathscr{P}(H^+)} \A_{H^+}(x) + L \, \|\partial_s H\|_{\infty},
\end{equation}
where $L$ is the length of an interval outside of which $\partial_s H(\cdot,t)$ vanishes for all $t\in \T$. 
This shows that if $c$ is at least the quantity on the right-hand side of inequality~\eqref{unifen}, 
then $\cu(H,J)=\cu_c(H,J)$, and concludes the proof of~(i).  
\end{proof}

The other crucial fact that we need is the following result, which implies in particular that $\cu(H,J)$ is not empty.

\begin{prop}
\label{surj}
Let $H = \{H^s\}$ and $J=\{J^s\}$ be as above. For every $z
\in \R \times \T$ and every $m \in M$ there is at least one $u\in \cu(H,J)$ such that $u(z)=m$.
\end{prop}

\begin{proof}[Outline of the proof.] The proof uses arguments from \cite[\S 6.4]{hz94}. 
Given a large positive number $T>0$, we can glue two disks to the cylinder $[-T,T]\times \T$ 
and obtain a sphere~$S_T$. The Floer equation for the pair $(H,J)$ on $[-T,T]\times \T$ can be extended 
to the two capping disks by homotoping the Hamiltonian to zero and by extending~$J$ by $J^-$ respectively~$J^+$ 
(see \cite[p.\ 231]{hz94}). This leads to spaces $\cu_T(H,J,z,m)$ of solutions~$u$ of this Floer equation on~$S_T$ with the property that $u(z)=m$. By the same argument sketched in the proof of Proposition~\ref{comp}, 
this space is compact in $C^{\infty}(S_T,M)$. It suffices to show that $\cu_T(H,J,z,m)$ is not empty 
for all large~$T$, since then any sequence $u_n \in \cu_{T_n}(H,J,z,m) $ with $T_n \rightarrow \infty$ 
has a subsequence which converges on compact sets to some $u \in \cu(H,J)$ such that $u(z)=m$, 
again by the usual compactness argument.

The space of solutions $\cu_T(H,J,z,m)$ can be seen as the set of zeroes of a smooth section of a suitable smooth Banach bundle 
$\pi \colon E \rightarrow B$. Here, $B$ is the Banach manifold of $W^{1,p}$ maps from $S_T$ to~$M$ mapping $z$ to $m$, 
where $2<p<\infty$, and the fiber of~$E$ at $u \in B$ is a Banach space of $L^p$ sections. 
By homotoping the Hamiltonian~$H$ to zero and the $S_T$-dependent $\omega$-compatible almost complex structure~$J$ 
to an $S_T$-independent one~$J_0$, we obtain a smooth 1-parameter family of smooth sections 
\[
S \colon [0,1]\times B \rightarrow E
\]
such that $\cu_T (H,J,z,m)$ is the set of zeros of $S(1,\cdot)$, while the zeros of $S(0,\cdot)$ are 
$J_0$-holomorphic spheres $u \colon S_T \rightarrow M$ such that $u(z)=m$. The assumption that $[\omega]$ vanishes 
on~$\pi_2(M)$ guarantees that the only zero of $S(0,\cdot)$ is the map that is constantly equal to~$m$. The usual compactness argument implies that the inverse image $S^{-1}(0_E)$ of the zero-section~$0_E$ of~$E$ under~$S$ 
is compact in $[0,1]\times B$. Moreover, the Fredholm results from \cite[Appendix~4]{hz94} imply that 
for each $(t,u)$ in $S^{-1}(0_E)$ the fiberwise differential of $S(t,\cdot)$ at~$u$ is a Fredholm operator 
of index~0. Finally, the fiberwise differential of $S(0,\cdot)$ at the unique zero $u\equiv m$ is an isomorphism 
(see \cite[Appendix 4, Theorem 8]{hz94}). Therefore, the section~$S$ satisfies all the assumptions of 
Theorem~\ref{abstract} in the appendix, from which we conclude that $S(1,\cdot)$ has at least one zero.
\end{proof}

\begin{rem} \label{existence}
{\rm
Note that Propositions \ref{comp} and \ref{surj} imply that for any $H\in C^{\infty}(\T\times M)$ 
the Hamiltonian vector field~$X_H$ has 1-periodic orbits. Indeed, Proposition~\ref{surj} implies that 
$\cu(H,J)$ is not empty and Proposition~\ref{comp}~(ii) then gives the existence of a 1-periodic orbit.
}
\end{rem}

\begin{rem}
\label{rich}
{\rm
By arguing as in \cite[\S 6.4]{hz94} more can be proved: Given $z \in \R \times \T$, denote by
\[
\ev_{\-2 z} \colon C^{\infty}(\R \times \T,M) \to M, \qquad \ev_{\-2 z}(u) := u(z)
\]
the evaluation map at $z$. 
Denote by $\check{H}^*$ the Alexander--Spanier cohomology functor with $\Z_2$-coefficients. 
Then the restriction of $\ev_{\-2 z}$ to $\cu(H,J)$ induces an injective homomorphism 
in cohomology:
\begin{equation} \label{e:ASinj}
\bigl( \ev_{\-2 z} |_{\cu(H,J)} \bigr)^*  \colon 
H^*(M) \cong \check{H}^*(M) \to \check{H}^* \bigl( \cu(H,J) \bigr).
\end{equation}
See \cite[\S 3.3]{Ha18}.
This fact in particular implies that the restriction of $\ev_{\-2 z}$ to $\cu(H,J)$ is surjective, 
i.e., Proposition~\ref{surj} holds.
In the case of an $s$-independent Hamiltonian, the injectivity of the map~\eqref{e:ASinj}
leads to the proof of the degenerate Arnol'd conjecture for closed symplectically aspherical manifolds, 
see \cite[Chapter 6]{hz94}.
}
\end{rem}

\section{Construction of an action selector} \label{s:minimal}

Let $H \in C^{\infty}(\T\times M)$ be a Hamiltonian. We would like to define an action selector for~$H$.

\subsection{The definition} \label{ss:def}

Denote by $\ck (M)$ the set of functions $K \in C^{\infty}(\R \times \T \times M)$ 
such that $\partial_s K$ has compact support 
and by $\cj_\omega(M)$ the set of smooth families $J = \{J^s\}$ of $\omega$-compatible almost complex structures 
on~$M$ such that $\partial_s J$ has compact support. 
Let $\ck(H)$ be the subset of those $K \in \ck(M)$ for which $K^-=H$,
and abbreviate $\cd (H) = \ck(H) \times \cj_\omega(M)$.
For $(K,J) \in \cd(H)$ let $\cu(K,J)$ be the space of finite energy solutions of Floer's equation~\eqref{floer-s}
defined by $K$ and~$J$. 
Let $(K,J) \in \cd(H)$ and assume that $\partial_s K$ and $\partial_s J$ are supported in $[s^-,s^+] \times \T \times M$. 
If $u \in \cu(K,J)$, then on $(-\infty,s^-]$ the function $s \mapsto \A_H(u(s))$ is non-increasing and bounded. 
Therefore, the function
\[
a^-_H \colon \cu(K,J) \to \R, \qquad 
a^-_H (u) \,:=\, \lim_{s\to -\infty} \A_H(u(s)) = \sup_{s \in (-\infty,s^-]} \A_H(u(s)) ,
\]
is well-defined. Being the supremum of a family of continuous functions, 
the function~$a^-_H$ is lower semi-continuous. As such, it has a minimum on the compact space~$\cu(K,J)$.

\begin{defn} \label{def:AA}
Let $H \in C^{\infty}(\T\times M)$ and $(K,J) \in \cd(H)$. We set
\[
A^-(K,J) := \min_{u \in \cu(K,J)} a_H^-(u), \qquad A(H) := \sup_{(K,J) \in \cd(H)} A^-(K,J).
\]
\end{defn}
It follows from Proposition \ref{spec} below that $A(H)$ is finite.

\begin{figure}[h]   
 \begin{center}
  \psfrag{s}{$s$}  \psfrag{s-}{$s^-$}  \psfrag{s+}{$s^+$}  \psfrag{H}{$K^- \!=H$}  \psfrag{K+}{$K^+$}  \psfrag{K}{$K$}
  \leavevmode\includegraphics{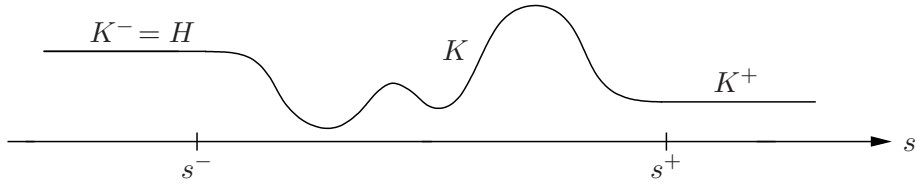}
 \end{center}
  \caption{A function $K$ deforming $H$, for $(t,x)$ fixed}   \label{fig.K}
\end{figure}

\subsection{First properties}  \label{ss:first}

Denote by
\[
\spec (H) := \{ \A_H(x) \mid x\in \mathscr{P}(H)\}
\] 
the set of critical values of $\A_H$. This set is compact, since $\mathscr{P}(H)$ is compact in~$C^{\infty}(\T,M)$ and $\A_H$ is continuous on $C^{\infty}(\T,M)$.  

Note that the number $A^-(K,J)$ belongs to $\spec (H)$. Indeed, take $u \in \cu (K,J)$ such that 
$a_H^-(u) = A^-(K,J)$. By Proposition~\ref{comp}~(ii), we find
$v$ in $\alpha$-$\lim (u)$, and $v$ is of the form $v(s,t)=x(t)$ with $x \in \cp(H)$ and $\A_H(x) = a_H^-(u)$. Hence $A^-(K,J)$ is a critical value of~$\A_H$.

Since $\spec (H)$ is compact, the supremum $A(H)$ is also a critical value of~$\A_H$. 
Therefore, we have proved the following result.

\begin{prop}[\bf Spectrality]
\label{spec}
$A(H)$ belongs to $\spec (H)$.
\end{prop}

Two very simple properties of the action selector $A$ are:
\begin{eqnarray}
\label{zero}
&A(H) = 0 \qquad & \mbox{if } H \equiv 0,\\
\label{e:r}
&A(H+r) =A(H)+\int_\T r(t)\,dt \qquad &
\forall \2 r \in C^\infty(\T), \; H \in C^{\infty}(\T\times M).
\end{eqnarray}
Indeed, the first property follows from the fact that for the Hamiltonian~$H \equiv 0$, 
the set~$\cp(H)$ consists of all the constant loops, which have action zero. 
The second property follows from the identities $\ck(H+r) = \ck(H)+r$ and
$a_{H+r}^-=a_H^- + \int_\T r(t)\,dt$. 
Less trivial is the following crucial result:

\begin{prop}[\bf Monotonicity]
\label{mon}
If $H_0,H_1 \in C^{\infty}(\T\times M)$ are such that 
$$
\int_{\T} \max_{x \in M}\, \bigl( H_1(t,x) - H_0(t,x) \bigr)\, dt \,\leq\, 0,
$$ 
then $A(H_1) \leq A(H_0)$. 
\end{prop}

\begin{proof} 
Fix $\varepsilon>0$. We shall prove that 
\begin{equation}
\label{mon1}
\sup_{(K_0,J_0) \in \cd (H_0)} \min_{\cu (K_0,J_0)} a_{H_0}^-  \,\geq\, 
\sup_{(K_1,J_1) \in \cd (H_1)} \min_{\cu (K_1,J_1)} a_{H_1}^- - \varepsilon,
\end{equation}
and the claim will follow from the arbitrariness of~$\varepsilon$. 
Proving~\eqref{mon1} is equivalent to showing that for every $(K_1,J_1)$ in~$\cd (H_1)$ 
there exists $(K_0,J_0)$ in $\cd (H_0)$ such that
\begin{equation} \label{mon2}
\min_{\cu (K_0,J_0)} a_{H_0}^- \,\geq\, \min_{\cu(K_1,J_1)} a_{H_1}^- - \varepsilon .
\end{equation}
Up to a translation, we may assume that 
\begin{equation} \label{e:K1H1}
K_1(s,t,x) \,=\, H_1(t,x)
\quad \mbox{and} \quad
J_1(s,t,x) \,=\, J_1^-(t,x) , 
\qquad \forall \2 s \leq 0.
\end{equation}
Let $\varphi \in C^{\infty}(\R)$ be a real function such that $\varphi' \geq 0$, 
$\varphi(s)=0$ for $s \leq 0$ and $\varphi(s)=1$ for $s\geq 1$. 
For $\lambda \in \R$ we define $K_0^{\lambda} \in \ck(H_0)$ by
\begin{equation} \label{e:fi}
K_0^{\lambda} (s,t,x) \,:=\, \varphi(s-\lambda) K_1(s,t,x) + \bigl( 1 - \varphi(s-\lambda) \bigr) H_0(t,x).
\end{equation}
We claim that there exists $\lambda \leq -1$ such that~\eqref{mon2} holds with $(K_0,J_0) = (K_0^{\lambda},J_1)$. 
Arguing by contradiction, we assume that for every $\lambda \leq -1$ there is a $u_{\lambda}$ 
in~$\cu(K_0^{\lambda},J_1)$ such that
\begin{equation} \label{mon3}
a_{H_0}^-(u_{\lambda}) \,<\, \min_{\cu(K_1,J_1)} a_{H_1}^- - \varepsilon.
\end{equation}
Let $(\lambda_n) \subset (-\infty,-1]$ be such that $\lambda_n \to -\infty$. 
By Proposition~\ref{comp}~(i), $\cu(K_1,J_1)$ is compact.
Arguing by a diagonal sequence argument,
we see that after replacing $(\lambda_n)$ by a subsequence, $(u_{\lambda_n})$ 
converges to some~$u$ in~$\cu(K_1,J_1)$.

\begin{figure}[h]   
 \begin{center}
  \psfrag{s}{$s$}  \psfrag{0}{$0$}  \psfrag{l}{$\lambda$}  \psfrag{l1}{$\lambda+1$}  \psfrag{H0}{$H_0$}  
  \psfrag{H1}{$H_1$}  \psfrag{K1}{$K_1$}
  \leavevmode\includegraphics{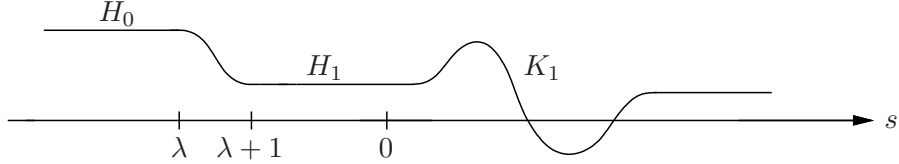}
 \end{center}
 \caption{The function $K_0^\lambda$, for $(t,x)$ fixed}  \label{fig.K0lambda}
\end{figure}

\noindent
%??? check later
We fix a number $s \leq 0$. 
If $\lambda_n \leq s-1$, then by~\eqref{e:K1H1} and the action-energy identity~\eqref{act-en},
\[
\begin{split}
& a_{H_0}^-(u_{\lambda_n}) \,\geq\, \A_{H_0} \bigl(u_{\lambda_n}(\lambda_n)\bigr)  
\\ 
& \,=\, \A_{H_1} \bigl(u_{\lambda_n}(s) \bigr) + \int_{[\lambda_n,s] \times \T} \bigl| \partial_\sigma u_{\lambda_n} \bigr|_{J_1^-}^2\, d\sigma\, dt -  \int_{[\lambda_n,s] \times \T} \varphi'(\sigma - \lambda_n) (H_1-H_0) (t,u_{\lambda_n}) \, d\sigma\, dt.
\end{split}
\]
By the hypothesis of the proposition and the fact that $\varphi'$ is non-negative we obtain the inequality
\[
\int_{[\lambda_n,s] \times \T} \varphi'(\sigma - \lambda_n) (H_1-H_0) (t,u_{\lambda_n}) \, d\sigma\, dt \leq \int_{[\lambda_n,s] \times \T} \varphi'(\sigma - \lambda_n) \max _{x\in M} (H_1-H_0) (t,x) \, d\sigma\, dt \leq 0,
\]
and hence the previous inequality gives us 
\[
a_{H_0}^-(u_{\lambda_n}) \geq \A_{H_1} \bigl( u_{\lambda_n}(s) \bigr).
\]
By taking the limit for $n\to \infty$, we deduce that
\[
\liminf_{n\to \infty} a_{H_0}^-(u_{\lambda_n}) \geq \A_{H_1}\bigl(u(s)\bigr),
\]
and by taking the supremum over all $s \leq 0$,
\[
\liminf_{n\to \infty} a_{H_0}^-(u_{\lambda_n}) \geq a_{H_1}^-(u).
\]
Together with \eqref{mon3}, this implies the chain of inequalities
\[
a_{H_1}^-(u) \leq \liminf_{n\to \infty} a_{H_0}^-(u_{\lambda_n}) \leq \min_{\cu(K_1,J_1)} a_{H_1}^- - \varepsilon,
\]
which is the desired contradiction because $u\in \cu(K_1,J_1)$.
\end{proof}

Monotonicity and property~\eqref{e:r} imply the following form of continuity.

\begin{prop} 
[\bf Lipschitz continuity]
\label{p:continuity}
For all $H_0, H_1 \in C^{\infty}(\T\times M)$ we have
$$
\int_\T \min_{x \in M}\, \bigl( H_1(t,x)-H_0(t,x) \bigr) \,dt \,\leq\, A(H_1)-A(H_0) 
\,\leq\, \int_\T \max_{x \in M}\, \bigl( H_1(t,x)-H_0(t,x) \bigr) \,dt .
$$
In particular, the action selector~$A$
is 1-Lipschitz with respect to the sup-norm on $C^{\infty}(\T\times M)$:
\[
\bigl| A(H_1) - A(H_0) \bigr| \,\leq\, 
\| H_1- H_0\|_\infty . 
\]
\end{prop}

\begin{proof}
Set
$$
c_-(t) \,=\, \min_{x \in M} \bigl( H_1(t,x)-H_0(t,x) \bigr), \quad
c_+(t) \,=\, \max_{x \in M} \bigl( H_1(t,x)-H_0(t,x) \bigr).
$$
Then 
$$
H_0(t,x) + c_-(t) \,\leq\, H_1(t,x) \,\leq\, H_0(t,x) + c_+(t), \qquad \forall \2 t \in \T,\, x \in M.
$$
Applying Proposition~\ref{mon} and~\eqref{e:r} we obtain
$$
A(H_0) + \int_\T c_-(t)\,dt \,\leq\, A(H_1) \,\leq\, A(H_0) + \int_\T c_+(t)\,dt
$$
as we wished to prove.
\end{proof}

\subsection{An equivalent definition} \label{ss:equivalent} 

By now, we know that our action selector~$A$ is spectral, monotone, and continuous. 
These properties already imply many further properties, see Proposition~\ref{p:formal} below,
and results like the unboundedness of Hofer's metric, see Section~\ref{ss:Hofer}. 
For most applications of an action selector, such as the non-squeezing theorem or the (almost)
existence of closed characteristics, one also needs that the selector is negative on functions
that are non-positive and somewhere negative.
To prove this property for our selector~$A$ we shall describe~$A$
by a minimax in which the space~$\cu (K,J)$ is replaced by 
a certain space of solutions of the Floer equation for~$H$.

Recall that $(\tau_{\sigma} u)(s) := u(\sigma+s)$.
Given $(K,J) \in \cd(H)$, consider the set 
\[
\cu_{\ess}(K,J) \,:=\, \set{ u \in C^{\infty}(\R\times \T,M)}{u = \lim_{n\to \infty} \tau_{s_n} u_n \mbox{ where } s_n \to -\infty \mbox{ and } (u_n) \subset \cu(K,J) }.
\]

\begin{ex}
Assume that neither $H$ nor $J$ depend on $s$. Then $\cu_{\ess}(H,J) = \cu (H,J)$.
\end{ex}

\proof
The inclusion $\cu_{\ess}(H,J) \subset \cu (H,J)$ holds because if $u_n$ belongs to $\cu(H,J)$ then also 
$(\tau_{s_n}u_n)$ does, and hence also $u = \lim_{n\to \infty} \tau_{s_n} u_n$ is in the same space, 
since $\cu(H,J)$ is closed.
Moreover, the inclusion $\cu(H,J) \subset \cu_{\ess} (H,J)$ holds because for $u \in \cu (H,J)$ we have 
$u_n := \tau_n u \in \cu (H,J)$
and $\displaystyle \lim_{n \to \infty} \tau_{-n}(u_n) =u$.
\proofend

As we shall see in Proposition~\ref{p:inv}, $\cu_{\ess}(K,J)$ is a compact $\tau$-invariant subspace of~$\cu(H,J^-)$. 
The space $\cu_{\ess} (K,J)$ is therefore
the space of those cylinders in~$\cu(H,J^-)$ which are essential with respect to~$K$, 
in the sense that they survive through the homotopy~$K$. 
We shall prove that the action selector
\[
A(H) \,=\, \displaystyle \sup_{(K,J) \in \cd(H)} \min_{\cu(K,J)} a_H^- \quad
\]
can be expressed as
\begin{equation} \label{eqdef}
A(H) \,=\, \sup_{(K,J) \in \cd(H)} \min_{\cu_{\ess}(K,J)} a_H,
\end{equation}
where $a_H$ is the continuous function
\[
a_H \colon C^{\infty}(\R\times \T,M) \rightarrow \R, \qquad a_H(u) := \A_H(u(0)).
\]
We begin with the following result.

\begin{prop} \label{p:inv}
The set $\cu_{\ess}(K,J)$ is a compact $\tau$-invariant subspace of $\cu(H,J^-)$. 
For every $z \in \R\times \T$ and $m\in M$ there exists $u \in \cu_{\ess}(K,J)$ such that $u(z)=m$.
\end{prop}

\begin{proof}
The inclusion $\cu_{\ess}(K,J) \subset \cu(H,J^-)$ is shown in the same way as the inclusion 
$\alpha$-$\lim (u) \subset \cu (H,J^-)$ in Proposition~\ref{comp}~(ii):
Let $u=\lim \tau_{s_n} u_n$ be an element of $\cu_{\ess}(K,J)$.
Since $v=\tau_{s_n} u_n$ solves the equation
\[
\partial_s v + (\tau_{s_n}J)(v) \bigl( \partial_t v - X_{\tau_{s_n} K}(s,t,v) \bigr) = 0,
\]
and since $\tau_{s_n} K$ converges to $K^-=H$ and $\tau_{s_n}J$ converges to~$J^-$, the map~$u$ 
is a solution of the $s$-independent Floer equation defined by~$H$ and~$J^-$. 
Moreover,
\[
E_{J^-}(u) \,\leq\, \liminf_{n \to \infty} E_{\tau_{s_n}J}(\tau_{s_n} u_n) \,=\, \liminf_{n\to \infty} E_J(u_n) \,\leq\, 
                      \sup_{v \in \cu(K,J)} E_J(v) \,<\, +\infty,
\]
where the finiteness of the last supremum follows from \eqref{unifen}.
Therefore, $\cu_{\ess}(K,J) \subset \cu(H,J^-)$.
If  $\sigma\in \R$, then
\[
\tau_{\sigma} u \,=\, \lim_{n\to \infty} \tau_{s_n + \sigma} u_n
\]
is in $\cu_{\ess}(K,J)$, which is therefore $\tau$-invariant. 
If
\[
v^h = \lim_{n \to \infty} \tau_{s_n^h} u_n^h, \qquad \mbox{where} \quad \lim_{n \to \infty} s_n^h = -\infty, \quad \forall \2 h \in \N,
\]
and $(v^h)$ converges to $v \in C^{\infty}(\R\times \T,M)$, then a standard diagonal argument implies the existence of 
a diverging sequence $(n_h) \subset \N$ such that
\[
\lim_{h \to \infty} \mathrm{dist}\, \bigl( \tau_{s_{n_h}^h} u_{n_h}^h,v^h \bigr) = 0, 
                                           \qquad \lim_{h\to \infty} s_{n_h}^h = -\infty,
\]
where $\mathrm{dist}$ is a distance on the metrizable space $C^{\infty}(\R\times \T,M)$. 
Therefore, $\tau_{s_{n_h}^h} u_{n_h}^h$ converges to~$v$, which hence belongs to~$\cu_{\ess}(K,J)$. 
This shows that $\cu_{\ess}(K,J)$ is a closed subspace of~$\cu(H,J^-)$. 
Since $\cu(H,J^-)$ is compact, so is $\cu_{\ess}(K,J)$.

Finally, given $z=(s,t)\in \R\times \T$, $m \in M$ and $n \in \N$, by Proposition \ref{surj} we can find 
$u_n \in \cu(K,J)$ such that $u_n(s-n,t)=m$. By compactness, a subsequence of $\tau_{-n} (u_n)$ converges to 
some $u\in \cu_{\ess}(K,J)$. Since
\[
\tau_{-n}(u_n)(z) = u_n(s-n,t) = m,
\]
we conclude that $u(z)=m$. 
\end{proof}

\begin{rem} \label{rich-ess}
{\rm
Actually, one can show that the space $\cu_{\ess}(K,J)$ satisfies the property of Remark~\ref{rich}: 
The restriction of $\ev_{\-2 z}$ to $\cu_{\ess}(K,J)$ induces an injective homomorphism in cohomology:
\[
\bigl( \ev_{\-2 z} |_{\cu_{\ess}(K,J)} \bigr)^*  \colon 
H^*(M) \cong \check{H}^*(M) \to \check{H}^* \bigl( \cu_{\ess}(K,J) \bigr).
\]
See \cite[Proposition 4.3.4]{Ha18}.
}
\end{rem}

Formula \eqref{eqdef} is an immediate consequence of the following result.

\begin{prop}
\label{cara}
$\displaystyle \min_{\cu(K,J)} a_H^- \,=\, \min_{\cu_{\ess}(K,J)} a_H$.
\end{prop}

\begin{proof}
Let $u \in \cu(K,J)$ be a minimizer of $a_H^-$. 
By Proposition~\ref{comp} (ii) there exists $v \in \mbox{$\alpha$-$\lim(u)$}$ with $v(s,t)=x(t)$ 
for some $x \in \cp(H)$, and
\[
a_H(v) \,=\, \A_H(x) \,=\, a_H^-(u) .
\]
Since $v \in \alpha$-$\lim(u) \subset \cu_{\ess}(K,J)$, we conclude
\[
\min_{\cu_{\ess} (K,J)} a_H \leq a_H(v) \,=\, a_H^-(u) \,=\, \min_{\cu(K,J)} a_H^-.
\]
Conversely, let $v \in \cu_{\ess}(K,J)$ be a minimizer of $a_H$. Then
\[
v = \lim_{n\to \infty} \tau_{s_n} u_n, \qquad \mbox{where } s_n \to -\infty \,\mbox{ and }\, (u_n) \subset \cu(K,J).
\]
Up to a subsequence, we may assume that $(u_n)$ converges to some $u \in \cu(K,J)$. 
For every fixed~$s$ belonging to a half-line $(-\infty,s^-]$ on which $\partial_s K$ vanishes, we have
\[
\A_H(u(s)) \,=\, \lim_{n\to \infty} \A_H(u_n(s)) \,\leq\, 
\lim_{n \to \infty} \A_H(u_n(s_n)) \,=\, \lim_{n\to \infty} a_H(\tau_{s_n} u_n) \,=\, a_H(v).
\]
By taking the limit for $s \to -\infty$, we find
\[
a_H^-(u) \,\leq\, a_H(v),
\]
which implies that
\[
\min_{\cu(K,J)} a_H^- \,\leq\, a_H^-(u) \,\leq\, a_H(v) \,=\, \min_{\cu_{\ess}(K,J)} a_H .
\]
\end{proof}

\subsection{Autonomous Hamiltonians}
Let $H \in C^{\infty}(M)$ be an autonomous Hamiltonian. 
In this case, the critical points of~$H$ are the constant orbits of~$X_H$, 
and in particular they are elements of~$\cp(H)$. 
In general, the vector field~$X_H$ can have other non-constant contractible orbits,  
but if this does not happen we can often calculate the value of the action selector~$A$.

\begin{prop} \label{p:two}
Let $H \in C^{\infty}(M)$ be an autonomous Hamiltonian with exactly two critical values. 
Assume also that $\cp(H)$ consists only of constant orbits.
Then 
\[
A(H) \,=\, \min_M H.
\]
\end{prop}

\begin{proof}
In this case, $\A_H$ has exactly two critical values, $\min H$ and $\max H$. 
Hence $A(H)$ is one of these two numbers.
For every $(K,J) \in \cd(H)$,
\begin{equation} \label{e:inequalitiesmin}
\min_{\cu_{\ess}(K,J)} a_H \,\leq\, \max_{\cu_{\ess}(K,J)} a_H \,\leq\, \max_{\cu(H,J^-)} a_H \,=\, \max_{\cp(H)} \A_H \,=\, \max_M H
\end{equation}
and, by Proposition~\ref{cara}, the number
\[
\min_{\cu_{\ess}(K,J)} a_H = A^-(K,J)
\] 
belongs to $\spec (H) = \{ \min H, \max H\}$.
Assume by contradiction that 
\[
A(H) = \sup_{(K,J) \in \cd H)} A^-(K,J)
\]
has the value $\max H$. Then we can find $(K,J) \in \cd (H)$ such that all
inequalities in~\eqref{e:inequalitiesmin} are equalities, and in particular
\[
\min_{\cu_{\ess}(K,J)} a_H = \max_M H.
\]
This identity implies that $\cu_{\ess}(K,J)$ consists only of constant cylinders defined by the maximum points 
of~$H$. Indeed, for every $u\in \cu_{\ess}(K,J)$ we can then find a periodic orbit $x\in \mathscr{P}(H)$ 
in the set $\alpha$-$\lim (u)$ with
\[
\max H \leq a_H(u) \leq \A_H(x).
\]
Since $\A_H(x) \leq \max H$ by the assumption, this yields $a_H(u) = \max H$.
By Proposition~\ref{p:inv}, the translates $\tau_s u$ also belong to $\cu_{\ess}(K,J)$,
whence $a_H(\tau_s u) = \max H$ for all $s \in \R$.
Since we also know that $\tau_s u \in \cu (H,J^-)$, 
it follows that $u$ is a trivial cylinder $u(s,t)=m$ with $H(m)=\max H$.

What we have just proved violates the surjectivity of the evaluation map $\ev_{\-2 z}|_{\cu_{\ess}(K,J)}$ from Proposition~\ref{p:inv}.
\end{proof}

\begin{rem} \label{twocrit}
{\rm
It is easy to construct autonomous Hamiltonians which satisfy the assumptions of the above proposition.
For instance, take a symplectically embedded ball $B \subset M$ of radius $3 \gve$
and a Hamiltonian~$H$ on~$M$ with support in~$B$ that on~$B$ is a radial function
$H = f(\pi |z|^2)$, where $f \colon \R_{\geq 0} \to \R_{\leq 0}$ is negative constant on $\{ r \leq \gve\}$,
vanishes on $\{ r \geq 2 \gve\}$, and has positive derivative on $\{\gve < r < 2 \gve\}$.
Then the minimum $f(0)$ and the maximum~$0$ are the only critical values of~$H$,
and if we further impose that $f' < 1$, we see as in the introduction that all 
non-constant periodic orbits of~$X_H$ have period larger than~$1$.
Hence Proposition~\ref{p:two} implies that $A(H)<0$.

By using Hamiltonians of this sort, together with the monotonicity property of~$A$, one can easily show that 
$A(H)<0$ for every non-positive Hamiltonian $H \in C^{\infty}(\T \times M)$ which is not identically zero. 
This is proved in Section \ref{s:axiom}, in which we investigate the properties of action selectors axiomatically.
}
\end{rem}

\begin{rem} \label{rem:explicit}
{\rm
In Remark~\ref{rem:ref} we give an explicit formula for $A(H)$ for a class of Hamiltonians different 
from the one in Proposition~\ref{p:two}. 
}
\end{rem}

%%%%%%%%%%%%%%%%%%%%%%%%%%%%%%%%%%%%%%%%%%%%%%%%%%%%%%%%%%%%%%%%%%%%%%%%%%%%%%%%%%%%%%%%%%%%%%

\section{An action selector on convex symplectic manifolds}
\label{s:convex}

A compact symplectic manifold $(M,\omega)$ is called convex if 
it has non-empty boundary and if near the boundary one can find a Liouville 
vector field~$Y$, namely such that $\cL_Y \omega = \omega$, 
which is transverse to the boundary and points outwards. 
We shall also assume that $[\omega]$ vanishes on~$\pi_2(M)$.

Since the boundary is compact, we can find $\gve >0$ such that the flow $\phi^t_Y$ 
of~$Y$ defines an embedding
\[
(1-\gve,1] \times \partial M \rightarrow M, \qquad (r,x) \mapsto \phi^{\log r}_Y(x)
\]
onto an open neighborhood $U$ of $\partial M$. 
This embedding defines a smooth positive function 
\[
r \colon U \rightarrow \R
\]
such that $r^{-1}(\{1\}) = \partial M$ and $r^{-1}((1-\gve,1)) = U \setminus \partial M$.

We consider the set $\ch (M)$ of smooth functions $H \colon \T \times M \to \R$ such that
$X_H$ has compact support in $\T \times (M \setminus \partial M)$.
In other words, for every component~$C_i$ of the boundary~$\partial M$
there exists an open neighborhood~$U_i$ of~$C_i$ in~$M$ and a function 
$h_i \in C^{\infty}(\T)$ such that $H(t,x) = h_i(t)$ for all $(t,x) \in \T \times U_i$.

The symbol $\mathscr{K}(M)$ now denotes the space of functions $K\in C^{\infty}(\R \times \T \times M)$ 
such that $X_K$ is supported in $\R \times \T \times (M\setminus \partial M)$ and $K(s,t,x)$ does not depend 
on $s$ for $s\leq s^-$ and for $s\geq s^+$, for some numbers $s^-,s^+$ depending on~$K$.

The set $\cj_\omega(M)$ now consists of all smooth families $J=\{J^s\}$ of $\omega$-compatible 
almost complex structures on~$M$ such that $\partial_s J$ is compactly supported,  
   $J(s,x)$ does not depend on~$s$   
for all $x$ in a neighborhood of~$\partial M$, 
and the equation
\begin{equation}
\label{contact}
dr \circ J = \imath_{Y} \omega
\end{equation}
holds on this neighborhood, where $r$ is the function which is induced by the 
Liouville vector field~$Y$ as above. 

For $K \in \ck (M)$ and $J \in \cj_\omega(M)$, let $\cu (K,J)$ be the set of finite energy solutions 
of the Floer equation~\eqref{floer-s} on~$M$. 
Being smooth maps defined on an open manifold (namely the cylinder $\R \times \T$), 
the elements $u\in \cu (K,J)$ are tangent to the boundary of~$M$ where they touch it. 
The next result implies, in particular, that the only elements $u \in \cu (K,J)$ 
that touch the boundary are constant maps.

\begin{lem} \label{le:conv}
Let $V'\subset U$ be an open neighborhood of $\partial M$ on which
the vector field $X_K(s,t,\cdot)$ vanishes for every $(s,t)\in \R \times \T$ and 
the almost complex structure 
   $J_0 := J(s,\cdot)$ is independent of~$s$ 
and 
satisfies~\eqref{contact}. Let $\delta>0$ be so small that the closure of the open set
\[
V := \{x\in U \;| \; r (x)> 1-\delta\}
\]
is contained in $V'$. Then the image of any $u \in \cu (K,J)$ that is not constant 
is contained in~$M\setminus V$.
\end{lem}

\begin{proof} 
The argument is well known, but we reproduce it here for the sake of completeness.
Let $u$ be an element of $\cu (K,J)$ and set $\Omega':=u^{-1}(V')$. The conditions on~$V'$ imply 
that $u|_{\Omega'}$ is a $J_0$-holomorphic map and $\rho := r\circ u \colon \Omega' \rightarrow \R$ 
is a subharmonic function.
The open set $\Omega := u^{-1}(V)$ satisfies $\overline{\Omega} \subset \Omega'$. 
We wish to prove that if the open set~$\Omega$ is not empty, 
then $u$ is a constant map.

The subharmonic function $\rho$ takes the value $1-\delta$ on $\partial \Omega$, 
and it is strictly larger than this value on~$\Omega$. 
By the maximum principle, $\Omega$ cannot have bounded components. 

Without loss of generality, we may assume that $s$ is unbounded from below 
on~$\Omega$. We claim that in this case $\Omega'$ contains a subset of the form 
$(-\infty,S)\times \T$, 
for some $S \in \R$. If this is not the case, we can find a sequence 
$(s_n,t_n) \in \R \times \T$ such that $s_n \rightarrow -\infty$ and $u(s_n,t_n) \in M \setminus V'$. 

As in the proof of Proposition~\ref{comp}, one shows that $\{ \tau_{s_n}u \mid n \in \N \}$ is relatively
$C^\infty_{\loc}$-compact in $C^\infty (\R \times \T, M)$.
Up to replacing $(s_n,t_n)$ by a subsequence, we can therefore assume that $\tau_{s_n}u$ converges in 
$C^\infty_{\loc}$ to a finite energy solution~$v$ of Floer's equation for~$J^-$ and~$K^-$,
and as in the proof of Proposition~\ref{comp}, we see that $v$ is a trivial cylinder for $K^-$,
that is $v(s,t) = x(t)$ is a 1-periodic orbit of~$X_{K^-}$.
In particular, the sequence of curves $u(s_n) = \tau_{s_n}u(0)$ converges to $x \in \mathscr{P}(K^-)$. 
Since all the solutions of $X_{K^-}$ through points in~$V'$ are constant, $x(\T)$ is disjoint from~$V'$. 
For $n$ large enough the set $u(\{s_n\}\times \T)$ is then contained in $M \setminus V$. 
This implies that $\{s_n\}\times \T$ is 
disjoint from~$\Omega$. Then the facts that $s_n \rightarrow -\infty$ and that $s$ is unbounded 
from below on~$\Omega$ force $\Omega$ to have bounded components. 
Since we have excluded this possibility, we reach a contradiction and conclude that 
$\Omega'$ contains a subset of the form 
$(-\infty,S) \times \T$, 
for some $S\in \R$.

Therefore, the biholomorphic map
\[
\varphi \colon \R \times \T \rightarrow \C \setminus \{0\}, \qquad 
   \varphi(s,t) = e^{2\pi(s+ it)},
\]
maps $\Omega'$ onto an open set of the form $\widetilde{\Omega}'\setminus \{0\}$, where $\widetilde{\Omega}'$ 
is an open neighborhood of the origin in~$\C$. Having finite energy, 
the $J_0$-holomorphic map $\tilde{u} := u \circ \varphi^{-1}$ extends holomorphically to~$\widetilde{\Omega}'$ 
by the removal of singularity theorem, see \cite[Theorem 4.1.2]{ms04}. Therefore, $\tilde{\rho} := r\circ \tilde{u}$ is a subharmonic function on $\widetilde{\Omega}'$. 
The fact that $s$ is unbounded from below on~$\Omega$ implies that $\tilde{\rho}(0) \geq 1-\delta$.

We now define the open subset $\widetilde{\Omega}$ of~$\C$ to be $\varphi(\Omega)\cup \{0\}$ if 
$\tilde{\rho}(0) > 1-\delta$ and $\widetilde{\Omega} := \varphi(\Omega)$ if $\tilde{\rho}(0) = 1-\delta$. 
The subharmonic function $\tilde{\rho}$ is strictly larger than $1-\delta$ on~$\widetilde{\Omega}$ and equal to
$1-\delta$ on its boundary. Hence the maximum principle implies that $\widetilde\Omega$ 
is unbounded, and so is a fortiori $\widetilde{\Omega}'$. By arguing as above and applying 
the removal of singularity theorem also at~$\infty$, we can extend 
the $J_0$-holomorphic map~$\tilde{u}$ to a $J_0$-holomorphic map~$\hat{u}$ which is defined 
on the open subset $\widehat{\Omega}' := \widetilde{\Omega}' \cup \{\infty\}$ of the Riemann sphere 
$\C \cup \{\infty\}$ such that $\hat{\rho} := r\circ \hat{u}$ satisfies $\hat{\rho}(\infty)\geq 1-\delta$.

As before, we set $\widehat{\Omega}$ to be $\widetilde{\Omega}\cup \{\infty\}$ if $\hat{\rho}(\infty) > 1-\delta$ and 
$\widetilde{\Omega}$ if $\hat{\rho}(\infty) = 1-\delta$. Then $\widehat{\Omega}$ is an open subset of 
the Riemann sphere, and the subharmonic function $\hat{\rho}$ is strictly larger than 
$1-\delta$ on it and equal to $1-\delta$ on its boundary. The maximum principle now forces 
$\widehat{\Omega}$ to be the whole Riemann sphere and $\hat{\rho}$ to be constant on it. 
In particular, $\hat{u}$ is a $J_0$-holomorphic sphere taking its values in~$V \subset U$. 
The fact that $\omega= d (\imath_Y \omega)$ is exact on~$U$ implies that $\hat{u}$ is constant, 
and so is~$u$.
\end{proof}

\begin{prop} 
\label{cvx}
The set $\cu(K,J)$ is compact in $C^{\infty}(\R \times \T,M)$. 
For every $z \in \R \times \T$ and $m \in M$ there exists $u \in  \cu(K,J)$ such that $u(z)=m$.
\end{prop}

\begin{proof}
Compactness is proved as in Proposition~\ref{comp}. Let $V$ be an open neighborhood
of~$\partial M$ satisfying the condition of Lemma~\ref{le:conv} for the pair~$(K,J)$ 
and for all elements of a smooth homotopy joining $(K,J)$ to~$(0,J_0)$, 
where the almost complex structure~$J_0$ does not depend on~$t$
and satisfies~\eqref{contact} on~$U$. If $m \in V$, then the constant map taking the value~$m$ belongs to 
$\cu(K,J)$. If $m \in M \setminus V$, then we can find a map $u \in  \cu(K,J)$ such that $u(z)=m$ 
arguing as in the proof of Proposition~\ref{surj}. 
Indeed, in this proof we may replace the closed manifold~$M$ by the open 
manifold~$M \setminus \partial M$ because the necessary compactness for the spaces of 
solutions~$u$ of the various Floer equations involved satisfying $u(z)=m$ 
is guaranteed by Lemma~\ref{le:conv}.
\end{proof}

Thanks to the above result, the action selector 
\[
A \colon \mathscr{H}(M) \rightarrow \R
\]
can be defined as in the closed case:
\[
A(H) := \sup_{(K,J) \in \mathscr{D}(H)} \min_{u\in \cu(K,J)} a_H^-(u),
\]
where $\mathscr{D}(H) := \mathscr{K}(H) \times \mathscr{J}_{\omega}(M)$, with $\mathscr{K}(H)$ denoting the set of all 
$K \in \mathscr{K}(M)$ such that $K^-=H$. The same properties that we have proved 
in the closed case hold also in the present setting.

\begin{rem}[\bf Exhaustions] \label{rem:exhaust}
{\rm
Consider a symplectic manifold $(M,\omega)$ that can be exhausted by compact convex symplectic manifolds: 
$M = \bigcup_{i=1}^\infty M_i$ where $M_1 \subset M_2 \subset \dots$ are compact convex submanifolds of~$M$. 
Also assume that $[\omega] |_{\pi_2(M)} =0$.
Examples are $(\R^{2n},\omega_0)$, cotangent bundles with their usual symplectic form 
and, more generally, Weinstein manifolds. 

Given a function $H \colon \T \times M \to \R$ with compactly supported Hamiltonian vector field~$X_H$, 
choose $i$ so large that the support of~$X_H$ is contained in the interior of~$M_i$.
Then $A(H;M_j)$ is well-defined for $j \geq i$.
These action selectors are sufficient for proving several results on exhaustions $(M,\omega)$,
like Gromov's non-squeezing theorem or the Weinstein conjecture for displaceable energy surfaces of contact type. 

Alternatively, one can define one single action selector for $(M,\omega)$ as follows.
While in general it is not clear whether the sequence $(A(H;M_j))$ stabilizes, or whether it is monotone, 
it is certainly bounded, as it takes values in the spectrum of~$H$. Hence we can define
\[
A(H) \,:=\, \liminf_{j \rightarrow \infty} A(H;M_j) \,\in\, \R .
\]
One readily checks that $A(H)$ is a minimal action selector on the space of functions $H \colon \T \times M \to \R$ 
with $X_H$ of compact support in the sense of Definition~\ref{def:axiom} below.
} 
\end{rem}

%%%%%%%%%%%%%%%%%%%%%%%%%%%%%%%%%%%%%%%%%%%%%%%%%%%%%%%%%%%%%%

\section{Axiomatization and formal consequences} \label{s:axiom}

It is useful to define an action selector by a few properties (``axioms'')
and to formally derive other properties from these axioms. 
In this way, it becomes clearer which properties of an action selector are fundamental 
and which other properties are just formal consequences of these fundamental ones.
The axiomatic approach also makes clear that properties that hold
for some action selectors, but do not follow from the axioms, 
rely on the specific construction of the selectors for which they hold. 
For example, the ``triangle inequality'' 
\[
\sigma (H_1\# H_2) \geq \sigma (H_1) + \sigma (H_2)
\]
and the minimum formula 
\[
\sigma (H_1+H_2) = \min \left\{ \sigma (H_1), \sigma (H_2) \right\}
\]
for functions supported in disjoint incompressible 
Liouville domains, both hold for the Viterbo selector and the PSS selector,
but are unknown for general minimal selectors.

\subsection{Axiomatization}
An attempt to axiomatize action selectors was made in~\cite{FrGiSch05},
and a very nice and slender set of four axioms was given in~\cite{HuLeSe15}.
We here give an even smaller list of axioms, that retains the first two axioms in~\cite{HuLeSe15},
but alters their non-triviality axiom and discards the minimum formula axiom.

Throughout this section we assume that $(M,\omega)$ is connected and symplectically aspherical 
(i.e.\ $[\omega] |_{\pi_2(M)} =0$).
If $M$ is closed, set $\ch (M) = C^\infty(\T \times M, \R)$,
and if $M$ is open (i.e.\ not closed), let $\ch (M)$ as in Section~\ref{s:convex}
be the set of functions in $C^\infty (\T \times M,\R)$ such that $X_H$ has compact support
in the interior of $\T \times M$.
The spectrum $\spec (H)$ of $H \in \ch (M)$ is again the set of critical values of the action functional~$\A_H$.

\begin{lem} \label{le:nowhere}
The spectrum $\spec (H)$ is a compact subset of $\R$ with empty interior.
\end{lem}

\proof
Since the support $S$ of $X_H$ is compact in $\Int (\T \times M) = \T \times \Int (M)$,
we find a compact submanifold with boundary $K \subset \Int (M)$ such that
$$
S \,\subset\, \T \times \Int (K) \,\subset\, \T \times K \,\subset\, \T \times \Int (M) . 
$$
It is well known that $\spec ( H |_{\T \times K})$ is compact and has empty interior.
It therefore suffices to show that $\spec (H) = \spec ( H |_{\T \times K})$.
The inclusion $\spec (H) \supset \spec ( H |_{\T \times K})$ is clear. 
So assume that $x$ is a 1-periodic orbit of~$X_H$ that is not contained in~$K$.
Since $H$ is locally a function of time on $\T \times (M \setminus K)$,
the orbit~$x$ is constant.
Since $M$ is connected, there exists $y \in \partial K$ such that $H(y,t) = H(x,t) = h(t)$ for all $t \in \T$.
Hence $\A_H(y) = \A_H(x) = \int_\T h(t)\, dt$.
\proofend

\begin{defn} \label{def:axiom}
{\rm
An action selector for a connected symplectically aspherical manifold $(M,\omega)$ is a map $\sigma \colon \ch (M) \to \R$
that satisfies the following two axioms.

\m \ni
{\bf A1 (Spectrality)}
$\sigma (H) \in \spec (H)$ for all $H \in \ch(M)$.

\m \ni
{\bf A2 ($C^\infty$-continuity)}
$\sigma$ is continuous with respect to the $C^\infty$-topology on $\ch(M)$.

\m \ni
An action selector is called minimal if, in addition, 

\m \ni
{\bf A3 (Local non-triviality)}
There exists a function $H \in \ch (M)$ with $H \leq 0$ and support in a symplectically
embedded ball in~$M$ such that $\sigma (H) <0$.
}
\end{defn}

\begin{rem}
{\rm
Assume that $\sigma \colon \ch(M) \to \R$ satisfies the spectrality axiom~A1.
Then $C^\infty$-continuity of~$\sigma$ is equivalent to $C^0$-continuity of~$\sigma$,
and continuity of~$\sigma$ implies its monotonicity, 
see Assertions~5 and~4 of Proposition~\ref{p:formal} below.
On the other hand, it is not clear if monotonicity of~$\sigma$, together with spectrality, 
implies its continuity, 
but this is so if $\sigma$ in addition has the shift property $\sigma (H+c)=\sigma(H)+c$ for all $H$ 
and $c \in \R$,
cf.\ the proof of Proposition~\ref{p:continuity}.
}
\end{rem}

Our selector $A$ on closed or convex symplectically aspherical manifolds
is indeed a minimal action selector, since it is spectral by Proposition~\ref{spec}, 
$C^\infty$-continuous since even Lipschitz continuous with respect to the $C^0$-norm by Proposition~\ref{p:continuity}, 
and non-trivial by Proposition~\ref{p:two} and Remark~\ref{twocrit}.
We note that the proof of monotonicity of~$A$ can be readily altered near the end to show directly that $A$
is $C^\infty$-continuous. 
In Proposition~\ref{p:formal} below we list many other properties of 
(minimal) action selectors,
some of which we have already verified
for $A$.

\subsection{Formal consequences}

For $H \in \ch (M)$ we abbreviate
$$
E^+(H) \,=\, \int_{\T} \max_{x \in M} H(t,x) \,dt , 
\qquad
E^-(H) \,=\, \int_{\T} \min_{x \in M} H(t,x) \,dt .
$$
The Hofer norm of $H$ is defined as
\begin{equation}  \label{e:Hofer}
\|H\| \,=\, E^+(H) - E^-(H) \,=\, \int_{\T} \left( \max_{x \in M} H(t,x) - \min_{x \in M} H(t,x) \right) dt .
\end{equation}
We also recall that the function 
\[
(H_1 \# H_2)(t,x) := H_1(t,x) + H_2(t, (\phi_{H_1}^t)^{-1}(x))
\]
generates the isotopy $\phi_{H_1}^t \circ \phi_{H_2}^t$.

A compact submanifold $U$ of $(M,\omega)$ is called a
{\it Liouville domain}\/ if $(U,\omega)$ is convex (see Section~\ref{s:convex} for the definition)
and if the corresponding Liouville vector field is defined on all of~$U$,
not just near the boundary~$\partial U$.
Examples are starshaped domains in~$\R^{2n}$ or 
fiberwise starshaped neighborhoods of the zero section of a cotangent bundle~$T^*Q$.
The domain~$U$ is {\it incompressible}\/ if the map $\iota_* \colon \pi_1(U) \to \pi_1(M)$ induced by inclusion is injective.
The above examples are incompressible Liouville domains.

Following \cite{Vi92} and \cite{HuLeSe15} we have
\begin{prop} \label{p:formal}
Assume that $(M,\omega)$ is connected and symplectically aspherical.
Then every action selector $\sigma$ on~$\ch (M)$ has the following properties.

\m \ni
{\bf 1. Zero:}
$\sigma (H)=0$ if $H \equiv 0$.

\m \ni 
{\bf 2. Shift:}
$\sigma (H + r) = \sigma (H) + \int_\T r(t)\,dt$ if $r \colon \T \to \R$ is a function of time.

\m \ni
{\bf 3. Coordinate change:}
If $\psi$ is a symplectomorphism of $(M,\omega)$ that is isotopic to the identity through symplectomorphisms,  
then 
$\sigma(H) = \sigma(H \circ \psi)$.

\m \ni
{\bf 4. Monotonicity:}
$\sigma (H_1) \leq \sigma (H_2)$\, if\,
$H_1 \leq H_2$.

\m \ni
{\bf 5. Lipschitz continuity:}
$E^- (H_1-H_2) \leq \sigma (H_1) - \sigma (H_2) \leq E^+(H_1-H_2)$.
In particular, $E^-(H) \leq \sigma (H) \leq E^+(H)$.

\m \ni
{\bf 6. Energy-Capacity inequality:}
$| \sigma (H_1) | \leq \| H_2 \|$
if $\phi_{H_2}$ displaces an open set $U \subset M$ such that $H_1$ is supported in~$\T \times U$.

\m \ni
{\bf 7. Composition:}
$\sigma (H_1) + E^-(H_2) \leq \sigma \left(H_1 \# H_2 \right) \leq \sigma (H_1) + E^+(H_2)$.

\m \ni
If, in addition, $\sigma$ is a minimal action selector, then:

\m \ni
{\bf 8. Non-degeneracy:}
If $H \leq 0$ and $H \neq 0$, then $\sigma (H) < 0$.

\m \ni
{\bf 9. Non-positivity:}
If $H$ has support in an incompressible Liouville domain,  
then $\sigma (H) \leq 0$.
In particular, $\sigma (H) =0$ for all non-negative Hamiltonians which are supported in an incompressible Liouville domain.
\end{prop}

\m
\ni
{\it Outline of the proof.}
Most properties are proved in~\cite[\S 3.1]{HuLeSe15}.
We focus on the new parts.
The first seven properties follow from the spectrality and the continuity axiom, 
together with the fact that the spectrum has empty interior.
This is immediate for Properties 1, 2,~3.
Properties 4,~5, 6 are proved in~\cite{HuLeSe15} for closed~$M$.
Their proof goes through for open~$M$ if one imposes the admissibility condition
in Lemma~21 of~\cite{HuLeSe15} only on a compact submanifold with boundary~$K \subset \Int (M)$
that contains the support of all the vector fields~$X_{H_s}$.
For Property~7 we compute, using Lipschitz continuity, 
\[
\sigma (H_1 \# H_2) - \sigma (H_1) \,\leq\, E^+(H_1 \# H_2 -H_1) \,=\, E^+(H_2 \circ (\phi_{H_1}^t)^{-1})
\,=\, E^+(H_2)
\]
and similarly 
\[
\sigma (H_1 \# H_2) - \sigma (H_1) \geq E^-(H_1 \# H_2 -H_1) = E^-(H_2).
\]
For the proof of Properties~8 and~9 we need two lemmas.
Let $U \subset M$ be a Liouville domain (the case $U=M$ is not excluded) and let $Y$ be the corresponding Liouville vector field.
Since $\overline U$ is compact and $Y$ points outwards along the boundary,
the flow $\phi^t_Y \colon U \to U$ of~$Y$ exists for all $t \leq 0$. 
The property $\cL_Y \omega = \omega$ of~$Y$
integrates to the conformality condition $(\phi_Y^t)^* \omega = e^t \omega$ for $t \leq 0$.
For each $\tau \leq 0$ define the Liouville subdomain
$U_\tau = \phi_Y^{\tau} (U)$, and for a Hamiltonian $H \colon \T \times M \to \R$ 
with support in~$U$ define the Hamiltonian
\[
H_\tau (t,x) \,:=\, 
\left\{
\begin{array} {ll}
e^\tau H (t,\phi_Y^{-\tau}(x)) & \mbox{if }\; x \in U_\tau, \\
0                         & \mbox{if }\; x \notin U_\tau.
\end{array}
\right.
\]
Then the support of $H_\tau$ lies in $U_\tau$.
The following lemma, that goes back to~\cite[proof of Prop.\ 5.4]{Pol14},
is taken from \cite[\S 3.2]{HuLeSe15}.

\begin{lem}  \label{le:Liouville}
Let $H \colon \T \times M \to \R$ be a Hamiltonian with support contained in 
a disjoint union of incompressible Liouville domains. Then
$\sigma (H_\tau) = e^\tau  \sigma(H)$ for all $\tau \leq 0$.
\end{lem}

For the proof one shows that $\spec (H_\tau) = e^\tau \spec (H)$, and so the claim follows from 
the spectrality and continuity axioms of~$\sigma$.

\begin{lem} \label{le:Gauto}
If $G$ is autonomous with $G \leq 0$ and $G \neq 0$, then $\sigma (G) <0$.
\end{lem}

\begin{proof}
Choose a non-empty open set $U \subset M$ such that $G |_{\overline U} < 0$.
Let $H$ and $B \subset M$ be a function and a symplectically embedded ball as in Axiom~A3, 
and let $0 \in B$ be the center of~$B$.
Take $x \in U$, and choose a Hamiltonian isotopy~$\psi$ of~$M$ with $\psi (0) =x$.
Then we find $\tau < 0$ such that $\psi (B_\tau) \subset U$.
Choosing $\tau$ smaller if necessary, we have $G \leq H_\tau \circ \psi^{-1}$.
Using Properties 3 and~4 and Lemma~\ref{le:Liouville} we obtain
\[
\sigma (G) \leq \sigma (H_\tau \circ \psi^{-1}) = \sigma (H_\tau) = e^{\tau} \sigma (H) <0.
\] 
\end{proof}

Property 8 now readily follows: 
Given $H \leq 0$ with $H \neq 0$ we find $t_0 \in (0,1)$ and $x_0 \in M$ with $H(t_0,x_0) < 0$. 
We can thus construct a function of the form $\alpha (t) \1 G(x)$ with $\alpha$ a non-negative
bump function around~$t_0$ and $G$ as in Lemma~\ref{le:Gauto} such that $H \leq \alpha G$.
Then $\sigma (H) \leq \sigma (\alpha G) = \sigma (c G) < 0$, where $c = \int_0^1 \alpha (t) \, dt >0$.
Indeed, the first inequality holds by monotonicity, and the last inequality by Lemma~\ref{le:Gauto}.
To see the equality $\sigma (\alpha G) = \sigma (cG)$, choose a smooth family of functions $\alpha_s(t)$, $s \in [0,1]$,
such that $\alpha_0(t) = \alpha(t)$, $\alpha_1(t) = c$ is constant, and $\int_0^1 \alpha_s(t)\, dt = c$ for all~$s$.
Then the Hamiltonian functions $H_s (t,x) := \alpha_s(t) \1  G(x)$ all generate the same time-1 map.
Lemma~\ref{le:sigmaindep} below combined with the shift property~2 now yield 
$\sigma (\alpha G) = \sigma (H_0) = \sigma (H_1) = \sigma (cG)$.
Let us give a direct and more elementary proof of $\sigma (H_0) = \sigma (H_1)$:
Since $X_{H_s} = \alpha_s(t) \1 X_G$, the time-1 orbits of~$H_s$ are reparametrisations 
of each other. 
Let $x_0$ be a 1-periodic orbit of $H_0$, and denote by $x_s$ the 1-periodic orbit of~$H_s$
with the same trace.
Then the area term $\int_{\D} \bar{x}_s^* \1 \omega$ of the action~$\A_{H_s}(x_s)$
does not depend on~$s$, since we can take the same disc for each~$s$, 
and the same holds for the Hamiltonian term
$$
\int_{\T} H_s(t,x_s(t))\, dt \,=\, \int_{\T} \alpha_s(t) \1  G(x_s(t)) \,dt \,=\, 
G(x_1(0)) \int_{\T} \alpha_s(t)\,dt \,=\, G(x_1(0)) \2 c
$$
since the autonomous Hamiltonian $cG$ is constant along its orbit~$x_1$.
It follows that $\spec (H_s)$ does not depend on~$s$.
Since this set has empty interior and $\sigma$ is continuous, $\sigma (H_s)$ neither depends on~$s$.

\s
We now prove Property~9. Let $U\subset M$ be a Liouville domain and choose $\varepsilon >0$ so small 
that there is a smooth function $F \colon M \to \R$ such that 
\[
F(x) = \varepsilon \quad  \mbox{on } U_{-2}, \qquad F(x) = 0 \quad \mbox{on  } M \setminus U_{-1}, 
\]
such that $\varepsilon$ and $0$ are the only critical values of $F$, and such that~$X_F$
has no non-constant 1-periodic orbits.
By spectrality, $\sigma (F) \in \{0,\varepsilon \}$.
Take $G$ as in Lemma~\ref{le:Gauto} with support in $M \setminus U_{-1}$ and such that $G \geq -\varepsilon$.
Then $F-\varepsilon \leq G$ and hence $\sigma (F-\varepsilon) \leq \sigma (G) < 0$.
Together with the shift property, $\sigma (F) = \sigma (F-\varepsilon)+\varepsilon < \varepsilon$, 
whence $\sigma (F) =0$.

Given $H \in \ch (U)$ we find $\tau <0$ so small that $H_\tau \leq F$. 
Then $\sigma (H_\tau) \leq \sigma (F) =0$ by monotonicity, and so $\sigma (H) = e^\tau \sigma (H_\tau) \leq 0$ by Lemma~\ref{le:Liouville}.
\proofend

\begin{figure}[h]   
 \begin{center}
  \psfrag{f}{$f$} \psfrag{G}{$G$} \psfrag{H}{$H_\tau$} \psfrag{e}{$\varepsilon$}
  \psfrag{-e}{$-\varepsilon$} \psfrag{U2}{$\partial U_{-2}$} \psfrag{U1}{$\partial U_{-1}$}  \psfrag{U}{$\partial U$}
  \leavevmode\includegraphics{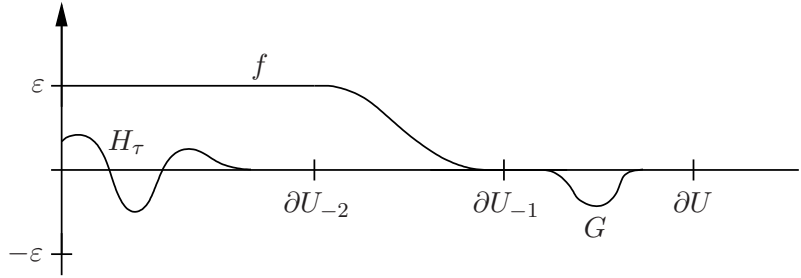}
 \end{center}
 \caption{The functions $f$, $G$, and $H_\tau$}   \label{fig.fGH}
\end{figure}

\subsection{Path independence}   \label{ss:path}
Let $\sigma \colon \ch (M) \to \R$ be an action selector.
Every function $H \in \ch (M)$ generates a Hamiltonian diffeomorphism~$\phi_H^1$.
Does $\sigma$ induce a map $\Ham (M,\omega) \to \R$ on the group formed by these diffeomorphisms?

If two functions in $\ch (M)$ differ by a constant, they have the same time-1 map.
In this paragraph we therefore restrict $\sigma$ to normalized functions:
If $M$ is closed, $H$ is normalized if $\int_M H(t,\cdot ) \,\omega^n =0$ for all $t \in \T$.
If $M$ is open, we fix an end~$e$ of~$\Int (M)$ and say that $H$ is normalized if 
for each $t \in \T$ the function $H(t,\cdot)$ vanishes on~$e$; 
notice that this normalization depends on the choice of~$e$.
Write $H_0 \sim H_1$ if $H_0,H_1$ are the endpoints of a smooth path~$H_s$ of normalized functions 
that all generate the same Hamiltonian diffeomorphism.

\begin{lem} \label{le:sigmaindep}
If $H_0 \sim H_1$, then $\sigma (H_0) = \sigma (H_1)$.
\end{lem}

\proof
The claim follows from the continuity of $\sigma$ if one knows that the sets $\spec (H_s)$
are independent of~$s$.
This in turn easily follows 
for closed~$M$ if the flow of~$H_0$ has a contractible 1-periodic orbit, see \cite[\S 3.1]{Sch00}, and 
for open~$M$ if the flow of~$H_0$ has a constant orbit of action zero, see \cite[Cor.\ 6.2]{FrSch07}.
The existence of such an orbit for closed~$M$
follows from Propositions~\ref{comp} and~\ref{surj}, see Remark~\ref{existence},
and for open~$M$ is obvious (take a point in the end~$e$ off the support of~$H_0$).
\proofend

The lemma implies that $\sigma$ descends from the set of normalized Hamiltonians to the universal cover $\widetilde \Ham (M,\omega)$,
where for open $M$ we denote by $\Ham (M,\omega)$ the group of Hamiltonian diffeomorphisms of~$M$ generated by normalized Hamiltonians. 
Does $\sigma$ further descend to $\Ham (M,\omega)$?
In other words, is it true that $\sigma (G) = \sigma (H)$ if $\phi_G = \phi_H$ for normalized $G , H$\,?
This is so if one knows that $\phi_G = \phi_H$ for normalized $G,H$ implies that $\spec (G) = \spec (H)$,
and that $\sigma$ satisfies the triangle inequality, see~\cite[proof of Prop.\ 7.1]{FrSch07}.
The first requirement always holds true.

\begin{lem} \label{le:specind}
Let $(M,\omega)$ be a symplectically aspherical manifold. 
If $\phi_G = \phi_H$ for normalized Hamiltonians $G,H \in \ch (M)$, then
$\spec (G) = \spec (H)$.
\end{lem}

\proof
This is again easy to verify for $M$ open~\cite[Cor.\ 6.2]{FrSch07}.
For $M$ closed the proof is more difficult.
Under the additional assumption that also the first Chern class of~$(M,\omega)$ vanishes on~$\pi_2(M)$, 
the proof is given by Schwarz~\cite[Theorem\ 1.1]{Sch00}.
One can dispense with this assumption thanks to results of McDuff~\cite{McDuff10}.
We give a rough outline of the argument. 

Let $\gamma$ be the loop in $\Ham (M,\omega)$ obtained by first going along $\phi_G^t$ 
and then along~$\phi_H^{1-t}$, for $t \in [0,1]$. 
To $\gamma$ one associates a bundle~$E$ with fiber~$M$ and base~$S^2$ by gluing two copies of the
trivial bundle $M \times D$ over the closed disk along their boundaries
via the loop~$\gamma$.
The total space~$E$ comes with a closed $2$-form~$\omega_E$, the so-called coupling form, 
that restricts to~$\omega$ on each fiber.
The assertion of the lemma will follow if we can show that 
\begin{equation} \label{e:S0}
\int_{S^2} s^* \omega_E \,=\, 0
\end{equation}
for one and hence any section $s \colon S^2 \to E$, see~\cite[Lemma~4.6]{Sch00}.

Let $J$ be an almost complex structure on~$E$ that is $\omega$-compatible on each fibre 
and such that the projection $E \to S^2$ is $J$-$i$-holomorphic, where $i$ is the usual complex
structure on~$S^2$.
Since $[\omega]$ vanishes on~$\pi_2(M)$, for a generic choice of~$J$ the space~$\mathcal{M}(J)$ of
holomorphic sections $s \colon S^2 \to E$ is a closed manifold, and its dimension is~$2n$.
Further, for given $z \in S^2$ the evaluation map 
$$
\mathcal{M}(J) \to M, \quad \ev_{z} (u) = u(z)
$$
has non-vanishing degree, see \cite[p.\ 117]{McDuff10}.
Now \eqref{e:S0} follows exactly as in the proof of Corollary~4.14 in~\cite{Sch00}.
\proofend

On the other hand, we do not know whether the triangle inequality holds for our action selector~$A$.
At least for Liouville domains, one can go around the triangle inequality 
and prove the following result.

\begin{prop} \label{p:[H]}
Assume that $(M,\omega)$ is a Liouville domain.
Then for any action selector~$\sigma$ on~$\ch (M)$ it holds that
$\sigma (G) = \sigma (H)$ whenever $G,H$ are normalized Hamiltonians with $\phi_G = \phi_H$.
\end{prop}

\proof
The claim is shown for $(\R^{2n},\omega_0)$ in~\cite[Proposition 11 in \S 5.4]{hz94}.
Their proof can be adapted to Liouville domains.
We give a somewhat streamlined argument. 

Let $L$ be a normalized Hamiltonian such that $\phi_L = \id$, that is, 
$\phi_L^t$, $t \in [0,1]$, is a loop in $\Ham (M,\omega)$.
The ends of $\Int (M)$ are in bijection with the components $N_1, \dots, N_k$ of the boundary of~$M$.
By assumption, $L(x,t) = h_j(t)$ for $x$ near~$N_j$, and one of these functions vanishes.
Denote the Liouville vector field on~$M$ again by~$Y$, and
for $\tau \leq 0$ set $M_\tau = \phi_Y^\tau (M)$. 
Then the function 
$$
L_\tau (t,x) \,=\, 
\left\{
\begin{array} {ll}
e^\tau L(t, \phi_Y^{-\tau}(x))& \mbox{if }\; x \in M_\tau, \\ [0.1em]
e^\tau h_j(t)                 & \mbox{if }\; x \in \displaystyle \bigcup_{\tau < t \leq 0} \phi_Y^t(N_j)
\end{array}
\right.
$$    
is smooth and normalized, and it generates the loop
$$
\phi_{L_\tau}^t (x) \,=\, 
\left\{
\begin{array} {ll}
\phi_Y^{\tau} \circ \phi_{L}^t \circ \phi_Y^{-\tau} (x) & \mbox{if }\; x \in M_\tau, \\
x                         & \mbox{if }\; x \notin M_\tau.
\end{array}
\right.
$$
Assume now that $\phi_G = \phi_H$. Let $G^-(t,x) = -G(t,\phi_G^t(x))$ be the function generating $\phi_G^{-t}$. 
Then $L := G^- \# H$ generates the loop $\phi_G^{-t} \circ \phi_H^t$.
Now note that 
$$
H \,\sim\, G \# G^- \# H \,=\, G \# L \,\sim\, G \# L_\tau .
$$ 
Together with Lemma~\ref{le:sigmaindep} and Property~7 we obtain
$$
\sigma (H) \,=\, \sigma (G \# L_\tau) \,\leq\, \sigma (G) + E^+(L_\tau) \,=\, \sigma (G) + e^\tau E^+(L)
$$ 
for every $\tau \leq 0$.
Hence $\sigma (H) \leq \sigma (G)$. In the same way, $\sigma (G) \leq \sigma (H)$.
\proofend

\begin{cor} \label{c:2m}
The conclusion of Proposition~\ref{p:[H]} also holds for all compact
2-dimensional symplectic manifolds $(M,\omega)$ with $M$ not diffeomorphic to the 2-sphere.
\end{cor}

\proof
If $(M,\omega)$ is closed and different from the sphere, then the fundamental group of $\Ham (M,\omega)$
is trivial~\cite[\S 7.2.B]{Pol01}, and the claim follows.
If $M$ is not closed, 
then $(M,\omega)$ is a Liouville domain, see \cite[Exercise 3.5.30]{McSa17},
and so the claim follows from Lemma~\ref{p:[H]}. 
\proofend

%%%%%%%%%%%%%%%%%%%%%%%%%%%%%%%%%%%%%%%%%%%%%%%%%%%%%%%%%%%%%%%%%%%%%%%%%%%%%

\section{Three applications of the existence of a minimal action selector}  \label{s:3app}

In this section we illustrate by three examples 
how the existence of a minimal action selector provides short and elementary proofs of theorems in 
symplectic geometry and Hamiltonian dynamics.
Our examples are Gromov's non-squeezing theorem, 
the existence of periodic orbits near displaceable energy surfaces, 
and the unboundedness of Hofer's metric. 

\subsection{Gromov's non-squeezing theorem}

In the standard symplectic vector space $(\R^{2n}, \omega_0)$ with $n \geq 2$ and $\omega_0 = \sum_{j=1}^n dx_j \wedge dy_j$
we consider the open ball $B^{2n}(r)$ of radius $r$ and 
the cylinder $Z^{2n}(R) = B^2(R) \times \R^{2n-2} = \{ x_1^2 + y_1^2 < R^2\}$.
For any $r>0$ the ball $B^{2n}(r)$ embeds into $Z^{2n}(R)$ by a volume preserving embedding;
just take a suitable diagonal linear map of determinant one.
Every symplectic embedding~$\varphi \colon B^{2n}(r) \to Z^{2n}(R)$ is volume preserving, since 
$$
\varphi^* (\omega_0^n) \,=\,(\varphi^* \omega_0)^n \,=\, \omega_0^n ,
$$
but there is no such embedding if $r>R$.
This celebrated theorem of Gromov~\cite{Gro85} shows that symplectic mappings are much more rigid than volume preserving mappings.

\begin{thm} \label{t:Gromov}
If $r>R$ there exists no symplectic embedding of the ball $B^{2n}(r)$ into the cylinder $Z^{2n} (R)$.
\end{thm}

\proof
Let $\varphi \colon B^{2n}(r) \to Z^{2n}(R)$ be a symplectic embedding.
Fix $\gve \in (0,r)$, and choose
a symplectic ellipsoid 
$$
E(R_1, \dots, R_n) = \Biggl\{ (x_1, y_1, \dots, x_n, y_n) \,\Bigg|\, \sum_{j=1}^n \frac{x_j^2 + y_j^2}{R_j^2} <1 \Biggr\}
$$
such that
\begin{equation} \label{e:BEZ}
\varphi \left( B^{2n}(r- \tfrac \gve 2 )\right) \,\subset\, E(R_1, \dots, R_n) \,\subset\, Z^{2n}(R) .
\end{equation}
By the elementary Extension after restriction principle from~\cite{EkHo89}, see also \cite[Appendix~A]{Sch05},
there exists a compactly supported Hamiltonian function~$G$ on $\T \times \R^{2n}$ such that
$\phi_G = \varphi$ on $B^{2n}(r-\gve)$.
Choose $\rho$ so large that the support of~$G$ is contained in $B^{2n}(\rho)$.
Let $\sigma$ be a minimal action selector on the convex symplectic manifold $(M,\omega) = (\overline B^{2n}(\rho), \omega_0)$.
For every open subset $U$ of the interior of~$M$ we set
$$
c_\sigma (U) \,=\, \sup \left\{ | \sigma (H)| \mid H \text{ has compact support in $\T \times U$} \right\} .
$$
We will prove that
\begin{eqnarray*}
\pi(r-\gve)^2
\,\stackrel{\raisebox{.5pt}{\tiny \textcircled{\raisebox{-.9pt} {1}}}}{\leq}\,
c_\sigma (B^{2n}(r-\gve)) 
& \stackrel{\raisebox{.5pt}{\tiny \textcircled{\raisebox{-.9pt} {2}}}}{=} &
c_{\sigma} (\phi_G (B^{2n}(r-\gve))) \\
& \stackrel{\raisebox{.5pt}{\tiny \textcircled{\raisebox{-.9pt} {3}}}}{\leq} &
c_\sigma (E(R_1, \dots, R_n)) 
\,\stackrel{\raisebox{.5pt}{\tiny \textcircled{\raisebox{-.9pt} {4}}}}{\leq}\, 
\pi (R_1 +\gve)^2
\,\stackrel{\raisebox{.5pt}{\tiny \textcircled{\raisebox{-.9pt} {5}}}}{\leq}\,
\pi (R +\gve)^2 .   % \notag
\end{eqnarray*}
Theorem~\ref{t:Gromov} then follows since $\gve>0$ can be chosen arbitrarily small.

\s
Choose a smooth function $f \colon [0,+\infty) \to \R$ with support in $[0,\pi(r-\gve)^2)$ such that 
\[
f(0) = \min f<0, \qquad f'(s)\in [0,1) \quad \forall \2 s \in  [0,+\infty), \qquad f(s) = 0 \quad \mbox{if } f'(s)=0.
\]
While all orbits of the Hamiltonian flow of the function
$H_f \colon \R^{2n} \to \R$, $H_f(z) = f (\pi |z|^2)$, are closed, only those on the sphere of radius~$\sqrt{s/\pi}$ 
with $f'(s) \in \Z$ have period one.
Hence the spectrum of~$H_f$ contains only~$0$ and~$\min H_f$.
By the non-degeneracy property~8 in Proposition~\ref{p:formal}, $\sigma (H_f) <0$ and so 
$\sigma (H_f) = \min H_f$ by the spectrality axiom.
Since we can choose $f$ such that $\min H_f$ is as close to $-\pi (r-\gve)^2$ as we like, 
inequality~\textcircled{\raisebox{-.9pt} {1}} follows.

Equality~\textcircled{\raisebox{-.9pt} {2}} follows from the coordinate change property~3 in 
Proposition~\ref{p:formal}.

Inequality~\textcircled{\raisebox{-.9pt} {3}} follows from the first inclusion in~\eqref{e:BEZ}:
We can use more Hamiltonian functions in $E(R_1, \dots, R_n)$ than in 
$\phi_G (B^{2n}(r-\gve)) \subset \varphi (B^{2n}(r- \tfrac \gve 2 )) \subset E(R_1, \dots, R_n)$.

It is easy to construct a compactly supported Hamiltonian function~$K_1$ on~$\R^2$
such that $\| K_1 \| \leq \pi(R_1 + \gve)^2$ and such that $\phi_{K_1}$ displaces~$B^2(R_1)$,
see \cite[p.\ 171]{hz94}. Let $K$ be a compactly supported cut-off of the function 
$(x_1, y_1, \dots, x_n,y_n) \mapsto K_1(x_1,y_1)$ such that $\phi_K$ displaces $E(R_1, \dots, R_n)$ and $\|K\|=\|K_1\|$.
Choosing $\rho$ larger if necessary, we can assume that the support of~$K$ is contained in~$B^{2n}(\rho)$.
The energy-capacity inequality~6 in Proposition~\ref{p:formal} now implies that 
$c_{\sigma}(E(R_1, \dots, R_n)) \leq \pi (R_1+ \gve)^2$.

Finally, the second inclusion in~\eqref{e:BEZ} shows that $\pi (R_1+ \gve)^2 \leq \pi (R+\gve)^2$.
\proofend

\subsection{Existence of periodic orbits near a given energy surface}

The search for periodic orbits of prescribed energy is a traditional topic of 
celestial mechanics and therefore also of Hamiltonian dynamics.
We consider an autonomous Hamiltonian $H \colon M \to \R$ on a symplectic manifold $(M,\omega)$,
and assume that $c$ is a regular value of~$H$ with compact energy surface $S_c = H^{-1}(c)$.
By preservation of energy, $S_c$ is invariant under the Hamiltonian flow of~$H$.
Examples by Ginzburg~\cite{Gi97} and Herman~\cite{He99} show that $S_c$ may carry no periodic orbit.
We therefore look for periodic orbits on nearby energy surfaces $S_{c'} = H^{-1}(c')$.

\begin{thm} \label{t:dense}
Let $(M,\omega)$ be a compact symplectically aspherical symplectic manifold, which is either closed or convex, 
and assume that the compact and regular energy surface $S_c = H^{-1}(c)$ is disjoint from $\partial M$ and 
displaceable, namely
there exists a smooth function $K \colon \T \times M \to \R$ 
with support in $\T \times (M \setminus \partial M)$ such that $\phi_K(S_c) \cap S_c = \emptyset$. 
Then there exists a sequence $c_j \to c$ of regular values of~$H$ 
such that every energy surface $S_{c_j}$ carries a periodic orbit of the flow of~$H$. 
\end{thm}

By applying this result to sufficiently large balls or disc bundles we obtain the existence 
of nearby periodic orbits for compact regular hypersurfaces in~$\R^{2n}$ and 
(under the displaceability assumption) in cotangent bundles.

\b \ni
{\it Proof of Thereom~\ref {t:dense}.}
Since $c$ is a regular value and $S_c$ is disjoint from $\partial M$, 
we find an open interval $I = (c-\varepsilon, c+\varepsilon)$ of regular values of~$H$ 
such that the union $U = \coprod_{c' \in I}S_{c'}$ of diffeomorphic hypersurfaces forms an open neighbourhood
of~$S_c$ in~$M\setminus \partial M$.
Choose a smooth function $K \colon \T \times M \to \R$ with support in $\T \times (M \setminus \partial M)$ 
such that $\phi_K(S_c) \cap S_c = \emptyset$.
Then $\phi_K$ displaces a whole neighbourhood of~$S_c$.
We can therefore choose $\varepsilon$ smaller if necessary, such that
$\phi_K (U) \cap U = \emptyset$. 
Let $f_\gve \colon \R \to \R$ be a smooth non-positive function with support in~$I$
whose only critical values are~$0$ and $-\|K\|-1$, see Figure~\ref{fig.fe}.

\begin{figure}[h]   
 \begin{center}
  \psfrag{R}{$\R$}  \psfrag{ce-}{$c-\varepsilon$}  \psfrag{ce+}{$c+\varepsilon$}  
  \psfrag{K}{$-\|K\|-1$} \psfrag{fe}{$f_{\varepsilon}$} 
  \leavevmode\includegraphics{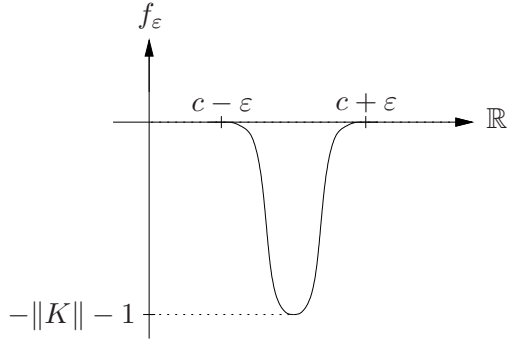}
 \end{center}
  \caption{The function $f_{\varepsilon}$}   \label{fig.fe}
\end{figure}

\noindent
The function $H_\gve := f_\gve \circ H$ has support in~$U$.
Let $\sigma$ be a minimal action selector for~$M$ (it exists by our construction 
in Sections~\ref{s:minimal} and~\ref{s:convex} and by Remark~\ref{rem:exhaust}).
By the non-degeneracy property~8 in Proposition~\ref{p:formal},
$\sigma (H_\gve) < 0$.
Further, the energy-capacity inequality~6 in Proposition~\ref{p:formal} 
shows that $|\sigma (H_\gve)| \leq \|K\|$.
Hence 
$$
-\|K\|-1 < \sigma (H_\gve) < 0 .
$$
Since the only critical values of $H_\gve$ are $0$ and $-\|K\|-1$, and since $\sigma (H_\gve)$ belongs 
to the spectrum of~$H_\gve$, it follows that $H_\gve$ has a non-constant $1$-periodic orbit in~$U$. 
A constant reparametrization of this orbit is a periodic orbit~$\gamma$ of~$H$, and $H(\gamma) \in I$.
Since $\gve >0$ was arbitrary, the theorem follows.
\proofend

Theorem~\ref{t:dense} can be improved in two directions:
First, an elementary additional argument shows that the set of energies $c' \in (c-\gve,c+\gve)$ at which 
the flow of~$H$ has a periodic orbit actually forms a set of full Lebesgue measure.
Secondly, assume in addition that near~$S_c$ one can find a Liouville 
vector field~$Y$ transverse to~$S_c$. In this case, $S_c$ is called of contact type.
Using the local flow $\phi_Y^t$ of~$Y$ we define another foliation $\coprod_{c' \in I}\widetilde S_{c'} =: \widetilde U$ with central leaf  
$\widetilde S_c = S_c$ by 
$$
\widetilde S_{c'} \,:=\, \phi_Y^{c'-c}S_c .
$$
We now look at the ``tautological'' function $\widetilde H \colon \widetilde U \to \R$ given by 
$\widetilde H(x) = c'$ if $x \in \widetilde S_{c'}$.
The restrictions of the Hamiltonian flow of $\widetilde H$ to~$S_{c}$ and $\widetilde S_{c'}$ are conjugate under $\phi_Y^{c'-c}$
up to a constant time-change. By Theorem~\ref{t:dense} we find $c'$ such that the flow of~$\widetilde H$
has a periodic orbit on $\widetilde S_{c'}$.
Hence the flow of~$\widetilde H$ also has a periodic orbit on~$S_{c}$.
A reparametrisation of this orbit is a periodic orbit of the flow of our original function~$H$.
This result proves a special case of the Weinstein conjecture on the existence of a periodic Reeb orbit on 
any closed contact manifold.
We refer to Sections~4.2 and 4.3 of~\cite{hz94} for detailed proofs of these improvements.

\subsection{Unboundedness of Hofer's metric}  \label{ss:Hofer}

By Darboux's theorem, every symplectic manifold $(M,\omega)$ locally looks like the standard symplectic vector space
of the same dimension, and so there are no local geometric invariants of $(M,\omega)$.
However, on the group $\Ham (M,\omega)$ of Hamiltonian diffeomorphisms there is a bi-invariant Finsler metric,
the so-called Hofer metric,
which is defined by
$$
d (\phi, \id) \,=\, \inf_{H} \|H\|,
$$
where $H$ varies over those $H \in \ch (M)$ with $\phi_H^1 = \varphi$ and where $\|H\|$
is the Hofer norm defined by~\eqref{e:Hofer}. 
The only difficult point in verifying that $d$ is indeed a metric is its non-degeneracy.
For closed or convex symplectically aspherical manifolds, this can be done by using any minimal action selector.
We leave this nice exercise to the reader. 

Note that in the above infimum we can ask the Hamiltonian~$H$ to be normalized 
as in Section~\ref{ss:path}, which for $M$ closed means that 
\[
\int_M H(t,\cdot)\, \omega^n \,=\, 0 \qquad \forall \2 t \in \T.
\]
Indeed, any Hamiltonian can be normalized by adding a suitable function of~$t$, and this operation neither affects the Hamiltonian vector field nor the Hofer norm.
%??? here wohl indent
Symplectic geometers use their metric intuition to prove results on the metric space $(\Ham (M,\omega), d)$,
which in turn help understanding the dynamics and the symplectic topology of the underlying manifold~$(M,\omega)$,
see for instance~\cite{Pol01}.  
A first question one can ask on a metric space is whether it has bounded diameter.
The following result was proven by Ostrover~\cite{Ost03}
under the assumption that also the first Chern class vanishes on~$\pi_2(M)$,
and by McDuff~\cite{McDuff10} without this assumption.
They both used the PSS selector.

\begin{thm}
Let $(M,\omega)$ be a closed symplectically aspherical manifold. 
Then the Hofer metric on $\Ham (M,\omega)$ is unbounded.
\end{thm}

\proof
Let $B \subset M$ be a symplectically embedded ball in~$M$, so small that there exists a Hamiltonian
diffeomorphism $h$ of~$M$ with $h (B) \cap B = \emptyset$. 
We can assume that $h$ is the time-1 map of an autonomous and normalized Hamiltonian~$H$.
Let $f \colon M \to \R$ be a function such that $f = 1$ on $M \setminus B$ and $\int_Mf \,\omega^n = 0$.

\begin{figure}[h]   
 \begin{center}
  \psfrag{f}{$f$}  \psfrag{h}{$h$} \psfrag{1}{$1$} \psfrag{B}{$B$}
  \leavevmode\includegraphics{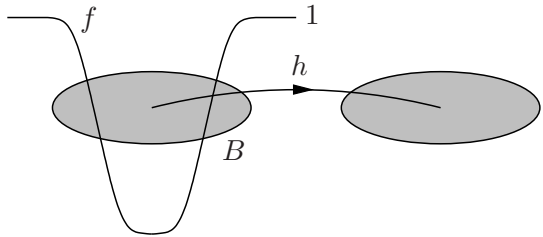}
 \end{center}
 \caption{The function $f$ and the map $h$}  
 \label{fig.1fh}
\end{figure}

For $s \in \R$ consider the Hamiltonian diffeomorphism
$$
\phi_s \,=\, h \circ \phi_{sf} \,=\, h \circ \phi_f^s .
$$
Let $G_s$ be any normalized Hamiltonian generating $\phi_s$.
We shall prove that
\begin{equation} \label{e:specGH}
\spec (G_s) \,=\, \spec (H) + s.
\end{equation}
Now let $\sigma$ be any (not necessarily minimal) action selector on~$\ch (M)$.
Since $\spec H$ has empty interior and $\sigma$ is continuous, 
\eqref{e:specGH} implies that $\sigma (G_s) = s_0+s$ for some $s_0 \in \R$ and every $s \in \R$.
Further, since $G_s$ is normalized, 
$E^-(G_s) \leq 0 \leq E^+(G_s)$ for every $s$.
Property~5 in Proposition~\ref{p:formal} thus implies that
$$
\| G_s \| \,=\, E^+(G_s) - E^-(G_s) \,\geq\, E^+(G_s) \,\geq\,\sigma (G_s) = s_0+s .
$$
This holds for all normalized Hamiltonians generating $\phi_s$, and so $d(\phi_s,\id) \geq s_0+s \to +\infty$ as $s \to +\infty$.

In order to prove \eqref{e:specGH} we use that by Lemma~\ref{le:specind} 
the set $\spec (G_s)$ does not depend on the specific choice of the normalized Hamiltonian~$G_s$ generating~$\phi_s$.
It therefore suffices to prove~\eqref{e:specGH} for the ``natural'' Hamiltonian generating $h \circ \phi_{sf}$
given by
$$
\widehat G_s(t,x) \,=\, \left\{
\begin{array} {ll}
\alpha (t+\frac 12) \2 s \2 f(x)         &\, \mbox{if }\, t \in \bigl[ 0, \frac 12 \bigr], \\ [0.3em] 
\alpha (t) \2 H(x)     &\, \mbox{if }\,  t \in \bigl[\tfrac 12, 1 \bigr] ,
\end{array}\right.
$$
that first generates the map $\phi_{sf}$ in time $\frac 12$ and then generates the map~$h$ in time~$\frac 12$,
yielding $h \circ \phi_{sf}$ in time~$1$.
Here, $\alpha \colon \R \to \R$ is a smooth non-negative function with support in~$(\frac 12,1 )$ and $\int_\R \alpha(t) \, dt =1$.
At first reading one should take $\alpha \equiv 2$, but this would result in a Hamiltonian~$\widehat G_s$ not
smooth at $t = \frac 12$.

Since $h(B) \cap B = \emptyset$, the contractible 1-periodic orbits of $\phi_{\widehat G_s}^t$
are exactly the contractible 1-periodic orbits of $\phi_{\alpha H}^t$:
Such an orbit $\gamma$ must start outside the ball~$B$ and does not move for $t \in \bigl[0,\frac 12 \bigr]$,
hence $f \equiv 1$ along $\gamma$.
The autonomous Hamiltonian~$H$ is also constant along~$\gamma$.
After reparametrization, $\gamma$ corresponds to a 1-periodic $\gamma_H$ of $\phi_H^t$.

Given such an orbit $\gamma$ and a disc $\overline \gamma$ that restricts to $\gamma$ along its boundary,
we compute the actions
\begin{eqnarray*}
\A_{\widehat G_s}(\gamma) &=&
\int_{\overline \gamma} \omega + s \int_0^\frac{1}{2} \alpha (t+\tfrac 12)\, f(\gamma(t)) \,dt 
                                         + \int_\frac{1}{2}^1 \alpha (t)\, H(\gamma(t)) \,dt \\
&=&
\int_{\overline \gamma} \omega + s + \int_0^1 H(\gamma_H(t)) \,dt \\
&=& \A_H(\gamma_H) +s .
\end{eqnarray*}
Claim \eqref{e:specGH} follows.
\proofend

%%%%%%%%%%%%%%%%%%%%%%%%%%%%%%%%%%%%%%%%%%%%%%%%%%%%%%%%%%%%%%%%%%%%%%%%%%%%%%

\section{Further directions and open problems} \label{s:further}

In this section we describe a few modifications of the construction in Section~\ref{s:minimal}.
For the proofs of the claims made we refer to~\cite{Ha18}.
To fix the ideas we assume that the symplectically aspherical manifold $(M,\omega)$ is closed.
 
\subsection{Smaller deformation spaces}
\label{s:smaller}

Our definition of an action selector admits several variations.

\subsubsection{Smaller classes of functions deforming $H$}
The set $\ck (H)$ is a large class of deformations of~$H$,
and it might be useful to consider smaller classes.

\begin{defn}
{\rm 
A subset $\ck' (M)= \bigcup_H \ck'(H)$ of $\ck(M)$ is {\it admissible}\/
if the following holds:
For any pair $H_0 \geq H_1$ and for any $K_1 \in \ck'(H_1)$ with 
$\supp \partial_s K_1\subset [s^-,s^+]$,
every $K_0 \in \ck (H_0)$ with $\partial_s K_0 \leq 0$ for $s \leq s^-$
and $K_0=K_1$ for $s \geq s^-$
belongs to~$\ck'(H_0)$.
}
\end{defn}

For every admissible set $\ck' (M) \subset \ck(M)$ and $\cd'(H) := \ck'(H) \times \cj_\omega(M)$,
$$
A'(H) \,=\, \sup_{(K,J) \in \cd'(H)} \min_{\cu (K,J)} a_H^-
$$  
defines a minimal action selector.

Examples of admissible sets are given by the monotone
decreasing deformations ($\partial_s K \leq 0$),
and for every real number~$c$ by the set $\ck_c (M) = \{ K \in \ck(M) \mid K^+ = c\}$.
Of course, $A' \leq A$ for every admissible subset $\ck'(M)$ of~$\ck(M)$.
For the classes $\ck_c(M)$ equality holds, see~\cite[Prop.\ 4.5.2]{Ha18}: 

\begin{prop} \label{p:down}
For every $c \in \R$ we have
$A_c(H) = A(H)$ for all $H \in C^\infty (\T \times M,\R)$.
\end{prop}

For functions $K \in \ck_0 (H)$,
the removal of singularity theorem
shows that the elements of~$\cu (K,J)$ are actually open disks
which are $J^+$-holomorphic near the origin and satisfy the Floer equation 
on a collar of the boundary equipped with cylindrical coordinates. 
These are exactly the objects which are used in the PSS~isomorphism from~\cite{PiSaSch94},
see Section~\ref{ss:PSS} below.

\subsubsection{Smaller classes of almost complex structures}
Given an $\omega$-compatible almost complex structure $J$ on~$M$ that does not depend on~$s$,
define
\[
A_J(H) \,=\, \sup_{K \in \ck (H)} \min_{\cu(K,J)} a_H^- . 
\]
While we do not know if $A_J(H)$ depends on $J$,
the number 
$\sup_{J} A_J(H)$
is of course independent of~$J$.
All of the functions $A_J(H)$ and $\sup_J A_J(H)$ on $C^\infty (\T \times M)$ are minimal action selectors, 
by the same (and sometimes easier) arguments as for~$A(H)$. 
We have chosen to give the construction for~$A(H)$ since this is more natural given
our deformation approach. 
Clearly, 
$$
A(H) \,\geq\, \sup_{J} A_J(H) \,\geq\, A_J (H) \quad \mbox{ for every $J$ and $H$.}
$$
Are these inequalities all equalities?
A class of Hamiltonian functions for which $A(H) = A_J(H)$ for every~$J$ is given in Remark~\ref{rem:ref}.
A somewhat different class is given by the intersection of the Hamiltonians in Proposition~\ref{p:two}
and the proposition below.

For the selectors $A_J$ and hence also for $\sup_J A_J$ we have
the following variant of Proposition~\ref{p:two}, 
which is Proposition~4.4.2 in~\cite{Ha18}. 
Its proof appeals to the transversality and gluing analysis from Floer theory.

\begin{prop} \label{p:morse}
Let $H \in C^{\infty}(M)$ be an autonomous Hamiltonian such that $X_H$ has no non-constant contractible 
closed orbits of period $T \in (0,1]$.
Then for every $\omega$-compatible $J$,
\[
A_J(H) \,=\, \min_M H.
\]
\end{prop}

\subsection{Action selectors associated to other cohomology classes} 

By using the result stated in Remark~\ref{rich}, one can construct spectral values
$A(\xi,H) \in \spec (H)$ for every non-zero cohomology class $\xi \in H^*(M;\Z_2)$. 
In the case $\xi = 1 \in H^0(M;\Z_2)$, the value $A(1,H)$ agrees with~$A(H)$.
These spectral values are monotone and continuous in~$H$ and hence are action selectors, 
but for $\xi \neq 1$ they are in general larger than the action selector~$A$ and not minimal. 
For instance, for the generator~$[M]$ of~$H^{2n}(M;\Z_2)$ and for $C^2$-small autonomous Hamiltonians 
with exactly two critical values we have
\[
A([M],H) \,=\, \max_M H .
\]
We refer to \cite[\S 5]{Ha18} for the proofs and for further properties of these action selectors.

\subsection{Comparison with the PSS selector} \label{ss:PSS}
                         
Recall that in \cite{Sch00} and~\cite{FrSch07} the PSS selector was constructed
on closed and convex symplectically aspherical manifolds with the help of 
Floer homology.
While our selector~$A$ already has many applications to Hamiltonian dynamics and symplectic geometry, 
some of the applications of the PSS selector rely on additional properties, that we were not able to verify 
for the selector~$A$. 
One such property is the triangle inequality 
$$
\sigma_{\PSS} \left(G \# H \right) \,\geq\, \sigma_{\PSS} (G) + \sigma_{\PSS} (H) ,
$$
that is stronger than the composition property~7 in Proposition~\ref{p:formal}.
Proving the triangle inequality requires the compatibility of the selector with the pair of paints product.
The triangle inequality can be used, for instance, 
to define a bi-invariant metric on the group $\Ham (M,\omega)$ that in general is different from the Hofer metric, 
and to construct partial symplectic quasi-states~\cite{EnPo06, PoRo14}.
Another property of the PSS selector is the minimum formula from~\cite{HuLeSe15}:
Given $H_1$ and~$H_2$ with support in disjoint incompressible Liouville domains, 
$$
\sigma_{\PSS} (H_1+H_2) = \min \left\{ \sigma_{\PSS} (H_1), \sigma_{\PSS}(H_2) \right\} .
$$
It is shown in~\cite{HuLeSe15} that for any minimal action selector~$\sigma$ satisfying this
formula there is an algorithm for computing~$\sigma$
on autonomous Hamiltonians on surfaces different from the sphere.

All properties of the PSS selector would of course
hold for our selector~$A$ if we could show that they agree. 
The selectors $A$ and~$\sigma_{\PSS}$ both select ``essential'' critical values, 
but in a rather different way:
While $A(H)$ is the highest critical value of~$\A_H$ such that all strictly lower
critical points can be ``shaken off'',
$\sigma_{\PSS}(H)$ is the $\A_H$-action of the lowest homologically visible generator of
the Floer homology of~$H$.
Assuming that the reader is familiar with Floer homology,
we describe $\sigma_{\PSS}(H)$ in a way relevant for its comparison with~$A(H)$.

By the $C^0$-continuity of both selectors, we can assume that
all contractible $1$-periodic orbits of~$H$ are non-degenerate, in the sense
that for every such orbit~$x$, $1$ is not in the spectrum of the
linearized return map $d\phi_H^1 (x(0))$.
There are then finitely many $1$-periodic orbits of~$\phi_H^t$.
Fix $K \in \ck(H)$ such that $K^+=0$ and $J$ with $J^-$ generic.
Recall that for such functions~$K$, for every element $u \in \cu (K,J)$ 
the limit $\ev (u) := \lim_{s \to +\infty}u(s,t) \in M$ exists.
Choose a Morse function~$f$ on~$M$ with only one minimum~$m$,
and let $W^{\rm s}(m)$ be the stable manifold of~$m$
with respect to the gradient flow $-\nabla f$ of a generic Riemannian metric on~$M$.
Then $\sigma_{\PSS}(H)$ is the smallest action~$\A_H (x)$ of a contractible 1-periodic orbit~$x$
with the following properties:
$x$ is a generator of $\HF_0(H,J^-;\Z)$
(namely $x$ is in the kernel of the Floer boundary operator $\partial_{J^-}$ but not in its image, 
and $x$ has Conley--Zehnder index~$0$), 
and the number of those elements $u \in \cu(K,J)$ that start at~$x$ and satisfy $\ev (u) \in W^{\rm s}(m)$ is odd.
Then clearly $\sigma_{\PSS}(H) \geq \min_{\cu(K,J)}a_H^-$.
Since $\sigma_{\PSS}(H)$ does not depend on the choice of $K \in \ck_0(H)$ nor on~$J$,
we conclude that $\sigma_{\PSS}(H) \geq A_0(H)$. Together with Proposition~\ref{p:down}
we obtain the following result, which is Proposition~9.1.1 in~\cite{Ha18}. 

\begin{prop} \label{p:PSS}
$\sigma_{\PSS}(H) \geq A(H)$ for all $H \in C^\infty (\T \times M, \R)$.
\end{prop}

\begin{open}
Is it true that $A(H) = \sigma_{\PSS}(H)$ for all $H \in C^\infty (\T \times M, \R)$\,?
\end{open}

The following remark was made by the referee.

\begin{rem} \label{rem:ref}
The equality $A(H) = \sigma_{\PSS}(H)$ holds for $C^2$-small autonomous Hamiltonians~$H$,
and more generally for those $H$ with the following properties. 

\s
\begin{itemize}
\item[(H1)] 
There exists $m \in M$ such that
$H_t(m) =\min_{x \in M} H_t(x)$ for every $t \in \T$.
\item[(H2)] 
1 is not in the spectrum of the linearized flow 
$d \phi_H^t(m)$ for all $t \in (0,1]$.
%The Hessians $\hess (H_t)(m)$ with respect
%to an $\omega$-compatible Riemannian metric satisfy
%\[
%\left\| \hess (H_t)(m) \right\| \,<\, 2 \pi  \quad\, \text{for all }\, t \in \T .
%\]
\item[(H3)] 
The flow of $X_H$ has no non-constant contractible 
closed orbits of period $T \in (0,1]$.
\end{itemize}

\noindent
Indeed, for such Hamiltonians we have
\begin{equation} \label{ids3}
A (H) \,=\, \sigma_{\PSS}(H) \,=\, \int_\T H_t(m)\, dt .
\end{equation}
\end{rem}

\noindent
For the proof we shall show that
\begin{equation*} 
\int_\T H_t(m)\, dt \,=\, \sigma_{\PSS}(H) \,\geq\, A (H) \,\geq\, \int_\T H_t(m)\, dt .
\end{equation*}
The first inequality follows from Proposition~\ref{p:PSS}. 
Assumption~(H3) in particular implies that the only contractible 1-periodic orbits of the flow of~$X_H$
are the rest points.
Together with~(H1) we obtain $\int_\T H_t(m)\, dt = \min \spec (H)$,
whence the second inequality follows in view of the spectrality of~$A$.
For the equality 
\begin{equation} \label{e:HPSS}
\int_\T H_t(m)\, dt \,=\, \sigma_{\PSS}(H)
\end{equation}
we first notice that the assumptions (H1) and~(H3) imply that the constant
orbit~$m$ is a critical point of the action functional~$\A_{s H}$ for every
$s \in [0,1]$ and that for any other critical point $y$ of~$\A_{s H}$,
\begin{equation}  \label{minact}
\A_{s H} (y) \,\geq\, \A_{s H}(m)
\,=\, s \int_{\T} H_t(m) \,dt, \quad\, s \in [0,1].
\end{equation}
By the continuity of $\sigma_{\PSS}$ it suffices to prove~\eqref{e:HPSS} for a $C^0$-close Hamiltonian.
In view of the non-degeneracy assumption~(H2) we find a $C^2$-small perturbation such that the 
contractible 1-periodic orbits of the new~$H$ are non-degenerate and such that 
(H1), (H2), and~\eqref{minact} still hold for the same point~$m$.
(There now may be non-constant contractible 1-periodic orbits~$y$.)

In the above description of the PSS selector we then choose the Morse function~$f$ such that $m$ is the unique minimum
and the deformation~$K$ of the form $K = \beta (s) H$ with a cut-off function~$\beta$.
One can now show using~\eqref{minact} that for a generic choice of the path~$J_s$ 
the critical point~$m$ is indeed selected by~$\sigma_{\PSS}$, see the proof of Theorem~5.3 in~\cite{FrSch07}.

\medskip
We also remark that the inequality $A(H) \geq A_J(H)$ and the spectrality of~$A_J$
imply that on the above class of Hamiltonian functions, every action selector $A_J$ is also equal 
to the three quantities in~\eqref{ids3}.

\appendix
\section{}

In this appendix we prove the following existence result for zeroes of a section of a Banach bundle, 
which is used in the proof of Proposition~\ref{surj}. Results of this kind are well-known and widely 
used in nonlinear analysis. The proof uses standard ideas from degree theory for proper Fredholm maps. 

\begin{thm}
\label{abstract}
Let $\pi \colon E \rightarrow B$ be a smooth Banach bundle over the Banach manifold~$B$ and let
\[
S \colon [0,1] \times B \rightarrow E
\]
be a $C^2$ map such that $S(t,\cdot)$ is a section of $E$ for every $t\in [0,1]$. 
Assume that $S$ satisfies the following conditions:
\begin{enumerate}[\rm (i)]
\item The inverse image $S^{-1}(0_E)$ of the zero section $0_E$ is compact.
\item For every $(t,x)\in S^{-1}(0_E)$ the fiberwise differential of the section $S(t,\cdot)$ at $x$ 
is a Fredholm operator of index~$0$.
\item There exists a unique $x_0 \in B$ such that $S(0,x_0)\in 0_E$.
\item The fiberwise differential of $S(0,\cdot)$ at $x_0$ is an isomorphism.
\end{enumerate}
Then the restriction of the projection 
$[0,1] \times B \to [0,1]$
to $S^{-1}(0_E)$ is surjective. In particular, there exists at least one $x_1 \in B$ such that $S(1,x_1) \in 0_E$.
\end{thm}

\begin{proof}
In order to simplify the notation, we assume that the Banach bundle $E$ has a global trivialization 
$E\cong B\times Y$, where the Banach space~$Y$ is the typical fiber of~$E$. The bundle to which we 
applied the theorem in the proof of Proposition~\ref{surj} 
has a global trivialization, since its typical fiber is an $L^p$-space and the 
general linear group of $L^p$-spaces is contractible by a version of Kuiper's theorem, \cite{Mit70}.
By using such a trivialization, we write
\[
S(t,x) = (x,F(t,x))
\]
for a suitable $C^2$-map $F \colon [0,1] \times B \rightarrow Y$ that has the following properties:
\begin{enumerate}[(i')]
\item The inverse image $F^{-1}(0)$ of $0 \in Y$ is compact.
\item For every $(t,x) \in F^{-1}(0)$ the differential of the map 
   $F(t,\cdot)$ at~$x$ is 
a Fredholm operator of index~$0$.
\item There exists a unique $x_0 \in B$ such that $F(0,x_0)=0$.
\item The differential of $F(0,\cdot)$ at $x_0$ is an isomorphism.
\end{enumerate}

We wish to show that the restriction of the projection $[0,1] \times B \to [0,1]$ to $F^{-1}(0)$ is surjective.
By (ii') the differential of~$F$ at each $(t,x) \in F^{-1}(0)$ is a Fredholm operator of index~1. 
Since Fredholm operators of a given index form an open set, there exists an open neighborhood 
$U \subset [0,1] \times B$ of $F^{-1}(0)$ on which $F$ is a Fredholm map of index~1. 
Fix $p \in F^{-1}(0)$.
Since Fredholm maps are locally proper \cite[p.\ 862, (1.6)]{Sm65}, 
we find an open neighborhood $V(p)$ such that $\overline V(p) \subset U$ and $F |_{\overline V(p)}$
is a proper map. Since $F |_{\overline V(p)} \colon \overline V(p) \to Y$ is a proper $C^2$ Fredholm map
of index~$1$, 
the Sard--Smale theorem \cite{Sm65} implies that the set~$R(p)$ of regular values of $F |_{V(p)} \colon V(p) \to Y$
is residual in~$Y$ (actually, in \cite{Sm65} the Sard--Smale theorem is stated under the assumption that the domain of the map 
-- in our case $[0,1] \times B$ -- is second countable, 
but the proof consists in showing the above local statement for proper Fredholm maps, see also \cite[p.\ 1106]{qs72}).                                                                   
Here, if $V(p)$ intersects $\{0,1\} \times B$ we view $V(p)$ as a manifold with boundary 
$V(p) \cap (\{0,1\} \times B)$, and a boundary point is considered to be regular if it is regular for the restriction of~$F$ to the boundary.
By~(i') we find finitely many points $p_j \in F^{-1}(0)$ such that $F^{-1}(0) \subset \bigcup_j V(p_j) =:V$.
The set $R := \bigcap_j R(p_j) \subset Y$ is also residual, and $F|_{\overline{V}}$ is a proper map. 

By (iii') and~(iv'), the point $(0,x_0)$ belongs to $F^{-1}(0)$ and $F(0,\cdot)$ 
is a local $C^2$-diffeomorphism at~$x_0$. Since $x_0$ is the unique zero of $F(0,\cdot)$, 
up to reducing~$V$ we may assume that the restriction of~$F$ to $V \cap (\{0\} \times B)$ is a diffeomorphism 
onto an open subset of~$Y$ containing a ball of radius~$r_0$ centered in~$0$.

\begin{figure}[h]   
 \begin{center}
  \psfrag{0}{$0$} \psfrag{1}{$1$} \psfrag{t}{$t$} \psfrag{x}{$x$} \psfrag{B}{$B$}
  \psfrag{x0}{$x_0$} \psfrag{z}{$z_n$} \psfrag{V}{$V$} \psfrag{G}{$\Gamma_n$}
  \leavevmode\includegraphics{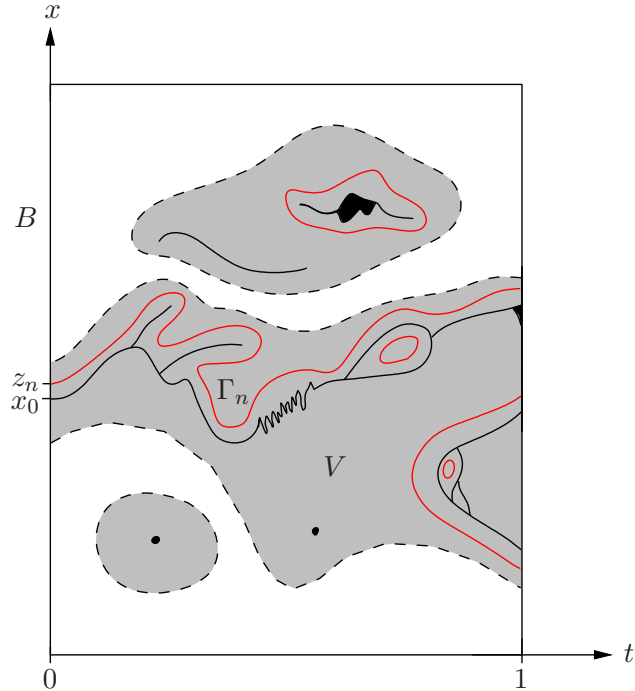}
 \end{center}
 \caption{The subsets $F^{-1}(0)$ and $\Gamma_n \subset \red{F^{-1}(y_n)}$ of $V \subset [0,1] \times B$.}  
 \label{fig.EBF}
\end{figure}

Denote by $\partial V$ the 
topological boundary of $V$ in $[0,1]\times B$. 
In Figure~\ref{fig.EBF} this set is indicated by the dashed curves.
Proper maps between metric spaces
are closed, so $F(\partial V)$ 
is a closed set and, since it does not contain~$0$, there exists a positive number~$r_1$ such that all the elements of $F(\partial V)$ have norm at least~$r_1$.

Altogether, for every natural number~$n$ 
we can find a regular value $y_n \in Y$ with 
\begin{equation} \label{e:yn}
\|y_n\| < \min \left\{ 2^{-n},r_0,r_1 \right\}.
\end{equation}
The set $F^{-1}(\{y_n\}) \cap V$ is a one-dimensional submanifold of $V$, and its boundary is precisely 
$F^{-1}(\{y_n\}) \cap V \cap (\{0,1\} \times B)$. Indeed, the fact that $\|y_n\| < r_1$ implies that 
$F^{-1}(\{y_n\})$ does not intersect the topological boundary $\partial V$ of~$V$ in $[0,1]\times B$. 
Together with the fact that $F|_{\overline{V}}$ is proper, this implies that the set $F^{-1}(\{y_n\}) \cap V$ 
is compact. 
Therefore, $F^{-1}(\{y_n\}) \cap V$ is a finite union of $C^2$-embedded images of $S^1$ and~$[0,1]$; 
embedded images of $[0,1]$ have boundary points on $V \cap (\{0,1\} \times B)$. 
The fact that $\|y_n\|<r_0$ and the property of the restriction of~$F$ to $V \cap (\{0\} \times B)$ 
stated above imply that $F^{-1}(\{y_n\})\cap V$ has exactly one point on $\{0\} \times B$, 
that we denote by $(0,z_n)$. Denote by $\Gamma_n$ the connected component of $F^{-1}(\{y_n\}) \cap V$ 
that contains~$(0,z_n)$. By what we have said above, $\Gamma_n$ is an embedded image of~$[0,1]$ 
with one boundary point $(0,z_n)$ and the other one on $\{1\}\times B$. By the connectedness of~$[0,1]$, 
the restriction of 
the projection $[0,1] \times B \to [0,1]$
to $\Gamma_n$ is then surjective. 

Now let $t \in [0,1]$ and let $u_n \in B$ be such that $(t,u_n)$ belongs to $\Gamma_n$. Then the sequence 
$(F(t,u_n)) = (y_n)$ tends to~$0$ by~\eqref{e:yn}, and by the properness of~$F$ on $\overline{V}$ the sequence $(t,u_n)$ has a subsequence which converges to some $(t,u)$. By the continuity of $F$ we have $F(t,u)=0$. 
This shows that $F^{-1}(0)$ intersects $\{t\}\times B$, as we wished to prove.
\end{proof}

%%%%%%%%%%%%%%%%%%%%%%%%%%%%%%%%%%%%%%%%%%%%%%%%%%%%%%%%%%%%%%%%%%%%%%%%%%%%%%%%%%%%%%%%%%%%%%%%%%%%
%\bibliographystyle{amsalpha}
%\bibliography{nonlinear}

\end{document}